\documentclass[11pt,a4paper]{scrartcl}
\usepackage[numbers,square]{natbib}
\usepackage{comment}
\usepackage{lipsum}
\usepackage{microtype}

\usepackage[font=small,labelfont=bf]{caption}
\usepackage
[
        a4paper,
        left=3cm,
        right=3cm,
        top=3cm,
        bottom=3cm,
]
{geometry}

\usepackage{setspace}
\usepackage{graphicx}
\graphicspath{{./figures}}

\usepackage{multicol,latexsym, array}
\usepackage{textcomp}
\usepackage[ansinew]{inputenc}
\usepackage{amsmath,amssymb, amsthm}
\usepackage[OT1]{fontenc}
\usepackage[english]{babel}
\usepackage{hyperref}

\hypersetup{
   colorlinks,
   linkcolor={blue},
   citecolor={blue},
   urlcolor={blue}
}

\usepackage{enumitem}
\usepackage{MnSymbol}
\usepackage{empheq}
\usepackage{rotating}
\usepackage{titletoc}
\usepackage{framed}
\usepackage{multirow,bigdelim}
\usepackage{subfigure}
\usepackage{bbm}
\usepackage[noend]{algpseudocode}
\usepackage{algorithm}
\usepackage{pgf}
\usepackage{pgfplots}
\usepackage{tikz}
\pgfplotsset{compat=1.17}
\usepackage{shadethm}
\usepackage{thmtools}
\usepackage{mdframed}
\usepackage{lipsum}
\usepackage{fixmath}
\usepackage{multicol}
\usepackage{footmisc}
\usepackage{macros}

\usepackage{xargs}  
\usepackage{xcolor} 

\definecolor{DarkGray}{RGB}{90,90,90}

\usepackage[colorinlistoftodos,prependcaption,textsize=tiny]{todonotes}
\usepackage{cleveref}

\date{\today}

    {
    \paragraph{Acknowledgements:} 
    }

\usepackage{mathrsfs}
\usetikzlibrary{arrows}

\usepackage{xpatch}

\makeatletter
\xpatchcmd{\endmdframed}
  {\aftergroup\endmdf@trivlist\color@endgroup}
  {\endmdf@trivlist\color@endgroup\@doendpe}
  {}{}
\makeatother

\usepackage{nicefrac}
\usepackage{dsfont}

\usepackage{bigints}
\usepackage{mathtools}
\usepackage{fixltx2e}

\usepackage{mathtools}
\usepackage{color}
\usepackage{fancyhdr}
\usepackage{lastpage}

\newcommand{\Cov}{\operatorname{Cov}}

\newcommand{\eps}{\varepsilon}

\newcommand{\Id}{\operatorname{Id}}

\newcommand{\vol}{\mathrm{vol}}
\newcommand{\dd}{\mathrm{d}}

\theoremstyle{plain}
\theoremstyle{remark}

\numberwithin{equation}{section}

\newcommand{\footremember}[2]{
	\footnote{#2}
	\newcounter{#1}
	\setcounter{#1}{\value{footnote}}
}

\newtheoremstyle{thmAB}%
  {6pt}{6pt}
  {\itshape}
  {}
  {\bfseries}
  {.}
  { }
  {\thmname{#1}\thmnumber{ #2}\thmnote{ \normalfont(#3)}}

\theoremstyle{thmAB}

\newtheorem{theoremA}{Theorem}

\newtheorem{theoremB}{Theorem}

\newif\ifexclude
\excludefalse  
\newcommand{\excludePart}[1]{\ifexclude \else #1 \fi}
\excludetrue

\begin{document}

\title{Large deviations for sums of multivariate stretched-exponential random variables:\\ the few-big-jumps principle}

\author{
  \begin{tabular}{ccc}
    \begin{tabular}{c}
      Nina Gantert\footremember{tum}{\scriptsize \;\;Technical University of Munich (TUM), Munich, Germany} \\[-1ex]
      \footnotesize{\href{mailto:nina.gantert@tum.de}{nina.gantert@tum.de}}
    \end{tabular}&
    \begin{tabular}{c}
      Joscha Prochno\footremember{passau}{\scriptsize \;\;University of Passau, Passau, Germany} \\[-1ex]
      \footnotesize{\href{mailto:joscha.prochno@uni-passau.de}{joscha.prochno@uni-passau.de}}
    \end{tabular}&
    \begin{tabular}{c}
      Philipp Tuchel\footremember{bochum}{\scriptsize \;\;Faculty of Mathematics, Ruhr University Bochum, Bochum, Germany} \\[-1ex]
      \footnotesize{\href{mailto:philipp.tuchel@ruhr-uni-bochum.de}{philipp.tuchel@ruhr-uni-bochum.de}}
    \end{tabular}
\end{tabular}\vspace{0.2cm}
}
\pagenumbering{arabic}
\date{}
\maketitle
\begin{abstract}
  \small Large deviations for sums of i.i.d.\ random variables with stretched-exponential tails (also called Weibull or semi-exponential tails) have been well understood since the 60's, going back to Nagaev's seminal work. Many extensions in the $1$-dimensional setting have been developed since then, showing that such deviations are typically governed by a single big jump. In higher dimensions, a corresponding theory has remained largely undeveloped. This work provides such a multivariate extension and establishes large deviation results for sums of i.i.d.\ random vectors in $\mathbb{R}^k$ under fairly general assumptions. Roughly speaking, for some $\alpha\in(0,1)$, the log-probability of one random vector divided by $x$ exceeding a threshold $t$ in all components behaves asymptotically, for large $x$, as $x^\alpha$ times a negative infimum of a function $\mathcal{J}$. We prove large deviation results for sums of i.i.d.\ copies, where the rate function is given by a minimization of at most $k$ summands of $\mathcal{J}$. This establishes a few-big-jumps principle that generalizes the classical $1$-dimensional phenomenon: the deviation is typically realized by \emph{at most} $k$ independent vectors. The results are applied to absolute powers of multivariate Gaussian vectors as well as to various other examples. They also allow us to study random projections of high-dimensional $\ell_p^N$-balls, revealing interesting insights about the appearance of light- and heavy-tailed distributions in high-dimensional geometry.

  \medspace
\end{abstract}

\excludePart{
\vspace{0.5cm}
\noindent \textit{Keywords}: 
\vspace{0.5cm}

\noindent \textit{MSC 2020 subject classification}: primary ; secondary
 }
\date{} 

\section{Introduction}

In probability, the theory of large deviations provides a quantitative description of rare events. It originated in Cram\'er's work on sums of i.i.d.\ random variables from 1938, where he proved that if the moment generating function of the summands is finite in a neighborhood of the origin, then the empirical mean satisfies a large deviation principle with an exponential speed and a convex rate function given by the Legendre transform of the logarithmic moment generating function (see \cite{cramer1938nouveau}). Later developments by Donsker and Varadhan and the subsequent formulation of the large deviation framework by Varadhan and others have elevated large deviations to a cornerstone of modern probability theory (see \cite{cramer1938nouveau}, \cite{donsker1975asymptotic}, \cite{varadhan1966asymptotic} and the book by \cite{dembo2009techniques}). In this Cram\'er regime, the probability that the empirical mean deviates from its expectation decays at an exponential speed, i.e., much faster than a straight forward application of the law of large numbers suggests. For heavy-tailed distributions, however, the moment generating functions do not exist and the classical Cram\'er theory does not apply. The limiting behavior in this case is entirely different and is typically governed by large but rare summands. In one dimension, the empirical mean of i.i.d.\ heavy-tailed variables typically deviates, since a single summand causes the deviation. By writing $a(x)\sim b(x)$ if $a(x)/b(x)\to1$ as $x\to\infty$, a random variable $Y$ with values in $\RR$ has a $1$-dimensional \emph{stretched-exponential tail} if there exist $c>0$ and $\alpha\in(0,1)$ such that
\[
\log \PP(Y\ge t)\sim -c\,t^\alpha \qquad (t\to\infty).
\]
These random variables are also called \emph{Weibull-like} or \emph{semi-exponential} random variables. Their tails do not decay exponentially fast, but faster than any inverse power. This tail behavior lies just below the light-tailed case: at the endpoint $\alpha=1$ the asymptotics become exponential, corresponding to a setting with finite exponential moments, and hence to the classical Cram\'er theorem. In this natural regime, it follows from Nagaev's result (\cite{Nagaev1969a,Nagaev1969b,nagaev1969integral}) that if $(Y_i)_{i\ge1}$ are i.i.d.\ centered with $Y_1$ having stretched-exponential tail as above, then for every $t>0$,
\[
\lim_{N\to\infty}\frac{1}{N^\alpha}\log \PP\left(\frac{1}{N}\sum_{i=1}^N Y_i\ge t\right)
=\lim_{N\to\infty}\frac{1}{N^\alpha}\log \PP(Y_1\ge Nt) 
= -c\,t^\alpha .
\]
This phenomenon is well known and often referred to as \emph{one-big-jump principle} or \emph{principle of a single big jump}. The upper tail of the sum is, on the logarithmic scale, the same as the tail of the maximum of the summands. Subsequent work has shown that this principle persists in a variety of settings. Denisov, Dieker, and Shneer proved that for general sub-exponential distributions one has $\PP(\sum_{i=1}^N X_i > t) \sim N\PP(X_1> t)$, $t\to\infty$, so that a rare event in the sum is caused by one large jump, see~\cite{Denisov2008}. Gantert, Ramanan, and Rembart established large deviations for certain weighted sums of stretched-exponential random variables, see~\cite{Gantert2014}. Lehtomaa later obtained a large deviation principle (LDP) for the sample mean of i.i.d.\ heavy-tailed variables with finite moments of all orders, see~(\cite{lehtomaa2017large}). Brosset et al. identified the transition speed between the Gaussian/moderate deviation regime and the sub-exponential one-big-jump regime for stretched-exponential tails, providing asymptotics across regimes and an explicit rate function exactly at the transition scale, see~\cite{brosset2022large}. Aurzada extended in \cite{Aurzada2020large} the weighted-sum theory to infinite weighted series of independent stretched-exponential variables, which generalizes~\cite{Gantert2014}. A comprehensive survey of large deviations for $1$-dimensional heavy-tailed sums may be found in~\cite{Denisov2008}.

Surprisingly, beyond the $1$-dimensional setting, much less is known, even though vectors with heavy-tailed components arise naturally in many applications ranging from insurance and finance to telecommunications and statistical mechanics. The case of independent components is easy to obtain from the $1$-dimensional setting, but dependence between components poses significant challenges. For heavy-tailed risk processes, H\"agele and Lehtomaa established multivariate large deviation principles ( see~\cite{hagele2021large}) under the assumption that $\bX = R\bU$, where $R$ is $1$-dimensional stretched-exponential and (in a certain sense) asymptotically independent of $\bU$, which is distributed on the unit sphere. In this setup, the problem can essentially be reduced to the $1$-dimensional setting by using the tails of $R$. However, a general description of large deviations for empirical means of stretched-exponential random vectors remained open. We found no unified definition of multivariate stretched-exponential tails in the literature. Nevertheless, we propose a natural notion that extends the $1$-dimensional definition and is applicable across a broad range of settings.

In this paper, we generalize the $1$-dimensional results to some multivariate random vectors with stretched-exponential tails. We consider an i.i.d.\ sequence $(\bX_i)_{i\ge1}$ of centered $\mathbb{R}^k$-valued random vectors whose marginal tails decay like $\exp(-c\,t^\alpha)$ for some $\alpha\in(0,1), c>0$, and whose joint tail is governed by a function $\calJ$, see Definition \ref{def:multiSE}. Our main result, Theorem \ref{thm:LDP-multi}, shows that if $x_N\to\infty$ with $x_N/N^{1/(2-\alpha)}\to\infty$, then the vector $\frac1{x_N}\sum_{i=1}^N \bX_i$ is (componentwise) larger than $\bt>0$ with probability $\sim\exp(-x_N^\alpha \calI_{\calJ}(\bt))$, where the rate function $\calI_\calJ$ is obtained by minimizing at most $k$ summands of $\calJ$. In contrast to the $1$-dimensional case, the deviation need not be realized by a single sample vector. At most $k$ summands contribute, each corresponding to a single vector deviation while each vector may contribute to the deviation of multiple coordinates. This \emph{few-big-jumps principle} extends the large deviation results for stretched-exponential random variables to the multivariate setting. We summarize the behavior of \(\frac1{x_N}\sum_{i=1}^N \bX_i\) depending on the growth of the scale \(x_N\) in \Cref{tab:scales} for the canonical regimes. Our main result covers the large deviation regime, while the scaling \(\sqrt{N} \ll x_N \ll N^{1/(2-\alpha)}\) is typically referred to as the moderate deviation regime (see \Cref{thm:LDP-multi} and \Cref{thm:MDP}). Our definition of multivariate stretched-exponential random variables is new. We obtain a condition on the density, which guarantees that a random vector satisfies our definition, and we apply our general theory to absolute powers of multivariate Gaussian vectors, multivariate Weibull distributions in the sense of \cite{hanagal1996multivariate}, multivariate generalized Gaussians and more. We also uncover interesting phenomena in high-dimensional geometry by applying our results to projections of $\ell_p^N$-balls, where we see some geometric analogues to light- and heavy-tailed distributions (see also \cite{kim2021large, prochno2024limit}).

The paper is organized as follows. We give the main definitions and results in Section \ref{sec: main-results}. Section \ref{sec: applications} presents applications, in particular to absolute powers of Gaussian vectors and to random projections of high-dimensional \texorpdfstring{$\ell_p^N$}{lpN}-balls. In Section \ref{sec: preliminaries}, we provide some preliminaries, and Section \ref{sec: Proofs} contains the proofs.

\begin{table}[h!]
\centering
\small
\setlength{\tabcolsep}{4pt}
\renewcommand{\arraystretch}{1.2}
\begin{tabular}{@{}p{3.3cm}p{3.8cm}p{1.2cm}p{6.5cm}@{}}
\hline\hline
\textbf{Regime} & \textbf{Scale for \(x_N\)} & \textbf{Speed} & \textbf{Limit / Rate Function}\\
\hline
Gaussian fluctuations &
\(x_N=\sqrt N\) &
-- &
\(\frac1{\sqrt N}\sum_{i=1}^N \bX_i \Rightarrow \mathcal N(\mathbf 0,\bSigma)\).\\[2pt]
Moderate deviations &
\(\sqrt N\ll x_N\ll N^{1/(2-\alpha)}\) &
\(x_N^2/N\) &
\(I(\bt)=\tfrac12\,\bt^\top\bSigma^{-1}\bt\).\\[2pt]
Large deviations &
\(N^{1/(2-\alpha)}\ll x_N\) &
\(x_N^\alpha\) &
\(\calI_{\calJ}(\bt)=\min_{\sum_{r=1}^k \bt^{(r)} = \bt}\sum_{r=1}^k \calJ\big(\bt^{(r)}\big)\).\\[2pt] 
\hline\hline
\end{tabular}
\caption{Fluctuation and deviation regimes for \(x_N^{-1}\sum_{i=1}^N \bX_i\) under Definition~\ref{def:multiSE} and Assumption~\ref{assu-LT}. Here \(\Cov(\bX_1) = \bSigma\).}
\label{tab:scales}
\end{table}

\section{Main results}\label{sec: main-results}

For $\alpha>0$, a function $f:\bbR^k\to[0,\infty]$ is called (positively) \emph{$\alpha$-homogeneous} if and only if $f(\lambda \bx)=\lambda^\alpha f(\bx)$ for all $\lambda > 0$. For $\bt=(t_1,\dots,t_k)$ and $\bs=(s_1,\dots,s_k)$ write $\bs\ge\bt$ if $s_i\ge t_i$ for all $i\in\{1,\dots,k\}$. We write $\bt^\alpha$ (in case of non-negative components) and $\max\{\bs,\bt\}$ for the componentwise power and maximum, respectively. By $\bt\odot\bs$ we denote the componentwise or Hadamard product.
For $B\subseteq \{1,\ldots, k\}$ and $\bx\in\bbR^k$ write $\bx_B$ for the vector $(x_i)_{i\in B}$. Let $\bbR^k_+ := \{\bx\in\bbR^k: \bx\ge 0\}$ and $\bt_+ := \max\{0,\bt\}$ for $\bt\in\bbR^k$. 
We write $\|\bx\|$ for the Euclidean norm of $\bx$ and $\bfm 1:=(1,\ldots,1)\in\bbR^k$.

\noindent We state the following natural assumption that will be used to ensure that the left tails are not heavier than the right tails.

\begin{assumption}
\label{assu-LT}
For $\bY = (Y_1,\dots,Y_k)$ and each $j\in \{1,\dots,k\}$ there exists $c_j>0$ such that for all large enough $t>0$,
\begin{equation}\label{LT}
\log\PP(Y_j\le -t) \le -c_j t^\alpha.
\end{equation}
\end{assumption}

\subsection{Large deviations and a few-big-jumps principle}\label{sec:large-deviations-for-right-tails}

We give the definition of multivariate stretched-exponential tails and state the main result. 

\begin{definition}[Multivariate stretched-exponential tails]
    \label{def:multiSE}
    Let $\alpha\in(0,1)$. A centered random vector $\bY=(Y_1,\dots,Y_k)$ with values in $\bbR^k$ has \emph{stretched-exponential tail} of rate $(\calJ,\alpha)$ if there exists a lower semicontinuous function $\calJ:\bbR_+^k\to[0,\infty)$ such that
    \[
    \lim_{x\to\infty}\frac{1}{x^\alpha}\log \PP\big(\bY \ge x\bt \big) = -\calJ(\bt_+)
    \]
    for all $\bt\in\bbR^k$ with at least one (strictly) positive component. Further, assume that $\calJ$ is non-degenerate in the sense that $\calJ(\bt)>0$ for all $\bt\in\bbR_+^k\setminus\{\bfm 0\}$.
\end{definition}

\noindent Often, $\calJ$ is given by the infimum of some function $\overline{\calJ}$, i.e. $\calJ(\bt)=\inf_{\bs\ge \bt}\overline{\calJ}(\bs)$. In the above definition, the function $\calJ$ is necessarily $\alpha$-homogeneous (see Lemma \ref{lem:homogeneity_barJ}) and componentwise nondecreasing. We further remark that $\bY$ satisfying Definition \ref{def:multiSE} and Assumption \ref{assu-LT} has finite second moments in each marginal (see Lemma \ref{lem:finite-second}). \\

\noindent For $\calJ:\bbR_+^k\to[0,\infty)$, we define the rate function $\calI_\calJ:\bbR_+^k\to[0,\infty)$ by
\begin{equation}\label{def-rate-fct}
\calI_\calJ(\bt)
:=\min\Big\{\sum_{r=1}^k \calJ\big(\bt^{(r)}\big):
\bt^{(r)}\in\bbR_+^k, \sum_{r=1}^k \bt^{(r)} = \bt\Big\}.
\end{equation}

\begin{theoremA}[Large deviations for sums of stretched-exponential random vectors]
\label{thm:LDP-multi}
Fix $\alpha\in(0,1)$. Let $(\bX_i)_{i\ge1}$ be i.i.d. random vectors with values in $\bbR^k$ and assume that $\bX_1$ satisfies \Cref{def:multiSE} with rate $(\calJ,\alpha)$ and Assumption \ref{assu-LT}. Let $(x_N)_{N\ge1}$ be a sequence of positive numbers such that $x_N\to\infty$ and $x_N N^{-1/(2-\alpha)}\to\infty$. Then, for any $\bt \in (0,\infty)^k$,
\[
\lim_{N\to\infty}\frac{1}{x_N^\alpha}\log \PP\left(\frac{1}{x_N}\sum_{i=1}^N \bX_i \ge \bt\right)
= - \calI_{\calJ}(\bt)
\]
\end{theoremA}

\noindent The above can be rewritten as 
\[
\PP\left(\frac{1}{x_N}\sum_{i=1}^N \bX_i \ge \bt\right) = \exp(-x_N^\alpha (\calI_\calJ(\bt)+o(1))).
\]

This theorem generalizes the $1$-dimensional result to the multivariate setting and for $k=1$, the theorem reduces to the classical $1$-dimensional result. From the rate function, it can be seen that a typical deviation is realized by at most $k$ independent deviations, where the size of each of these deviations is such that the variational problem in \eqref{def-rate-fct} is minimized. Each of the non-zero summands corresponds typically to one deviation of a random vector in the sum. Note however, that one vector may contribute to the deviation of multiple coordinates. In some situations, the deviation is realized by only one vector as in the $1$-dimensional case, but we give also examples where the deviation is typically caused by multiple vectors (see Lemma \ref{lem:two-big-jumps}). In the above setting, we can choose $x_N = N$ for any $\alpha \in (0,1)$.

\begin{remark}
The rate function $\calI_{\calJ}$ is not necessarily convex nor concave even in simple examples (see Lemma \ref{lem:I_not_convex_nor_concave_q_gt_2_rho_pos}). This contrasts with the light tailed case in Cram\'er's theorem, where the rate function is convex, and with the $1$-dimensional stretched-exponential case, where the rate function is concave.
\end{remark}

\subsection{Moderate deviation principles}

When $x_N$ grows slower than \(N^{1/(2-\alpha)}\), one is in the Gaussian or moderate deviation range. The following statement is known, but we could not find it in the literature as stated. We include it for completeness as this, together with Theorem \ref{thm:LDP-multi}, gives a full picture of the large deviations for sequences $(x_N)_{N\ge 1}$ with $x_NN^{-1/2}\to\infty$ and either $x_N N^{-1/(2-\alpha)}\to\infty$ or $x_N N^{-1/(2-\alpha)}\to 0$. For the case where $x_N \sim c N^{1/(2-\alpha)}$ for some $c>0$, the situation seems to be more involved. For the $1$-dimensional case, this case is understood and the rate function is given by a combination of the rate functions of the both regimes (see \cite{brosset2022large}). In the multivariate case, this question remains open.

For the definition of a large deviation principle, see Definition \ref{def: LDP} and for a result on Banach space valued random variables, see Proposition \ref{prop: MDP general}.

\begin{theoremB}[Moderate deviations]
\label{thm:MDP}
Let $(x_N)_{N\ge 1}$ satisfy $\sqrt{N}/x_N\to0$ and $x_N N^{-1/(2-\alpha)}\to0$. Let $\bX,\bX_1,\dots,\bX_N$ be i.i.d. centered with values in $\bbR^k$ and assume that for some $c>0$ and all large $t$,
\[
\PP(\|\bX\|\ge t)\le \exp(-c\,t^\alpha).
\]
If $\bSigma=\Cov(\bX)$ is invertible, then $\frac1{x_N}\sum_{i=1}^N \bX_i$ satisfies an LDP on $\bbR^k$ with speed $x_N^2/N$ and good rate function
\[
I(\bz)=\frac12\,\bz^\top \bSigma^{-1}\bz.
\]
In particular, for every $\bt\in(0,\infty)^k$,
\[
\lim_{N\to\infty}\frac{N}{x_N^2}\log \PP\left(\frac{1}{x_N}\sum_{i=1}^N \bX_i\ge \bt\right)
= - \inf_{\bz\ge \bt}\ \frac12\,\bz^\top \bSigma^{-1}\bz.
\]
\end{theoremB}

\begin{remark}
If $\bX$ satisfies \Cref{def:multiSE} with rate $(\calJ, \alpha)$ and Assumption \ref{assu-LT} then for large $t$,
\[
\PP(\|\bX\|\ge t)\le k\max_{1\le j\le k}\PP\big(|X_j|\ge \frac{t}{\sqrt{k}}\big)
\le k\max_{j}\big\{\PP\big(X_j\ge \frac{t}{\sqrt{k}}\big)+\PP\big(X_j\le - \frac{t}{\sqrt{k}}\big)\big\}
\le \exp(-c' t^\alpha)
\]
for some $c'>0$. Here the last inequality follows by Lemma \ref{lem:finite-second}. Thus, the hypothesis of \Cref{thm:MDP} holds.
\end{remark}

\section{Applications}\label{sec: applications}

We start with a condition on the density, which guarantees stretched-exponential tails in the sense of Definition \ref{def:multiSE} in the multivariate case.

\begin{lemma}\label{lem:density-implies-SE}
    Assume $\bY$ with $\EE[\bY]=0$ has Lebesgue density $f$ on $\bbR^k$.
    Assume further that there exist functions $p_1,p_2:\bbR_+^k\to(0,\infty)$, an $\alpha$-homogeneous, good rate function $\overline{\calJ}:\bbR_+^k\to[0,\infty)$ for some $\alpha\in(0,1)$, and $R>0$ such that, for all $\bx\in\bbR^k$ with $\|\bx_+\|\ge R$,
    \[
    p_1(\bx_+)\,e^{-\overline{\calJ}(\bx_+)}
    \ \le\
    f(\bx)
    \ \le\
    p_2(\bx_+)\,e^{-\overline{\calJ}(\bx_+)}.
    \]
    Assume, moreover, that for each compact set $K\subseteq \bbR_+^k$ and each $i\in\{1,2\}$,
    \[
    \lim_{x\to\infty}\frac{1}{x^\alpha}\log\Big(\sup_{\bs\in K} p_i(x\bs)\Big)=0.
    \]
    Then $\bY$ has stretched-exponential tail of rate $(\inf_{\bs\ge (\cdot)}\overline{\calJ}(\bs),\alpha)$ in the sense of Definition~\ref{def:multiSE}.
\end{lemma}

\subsection{Absolute powers of multivariate Gaussians}
We write $\mathcal{N}(\bfm 0,\bSigma)$ for the centered $k$-dimensional Gaussian distribution on $\bbR^k$ with covariance matrix $\bSigma$. Let \(\bY\sim\mathcal N(\bfm 0,\bSigma)\) on \(\bbR^k\) with \(\bSigma\) positive definite. Fix \(q>1\). Define \(\bZ=(|Y_1|^q,\ldots,|Y_k|^q)\).

\begin{lemma}\label{lem:power-gaussian-density}
 The random vector \(\bZ=(|Y_1|^q,\ldots,|Y_k|^q)\) defined above has density given by
    \[
    f_{\bZ}(\bz)=\frac{1}{(2\pi)^{k/2}(\det\bSigma)^{1/2}\,q^k\prod_{i=1}^k z_i^{(q-1)/q}}
    \sum_{\beps\in\{-1,1\}^k}
    \exp\Big(-\frac12\,(\beps\odot \bz^{1/q})^\top\bSigma^{-1}(\beps\odot \bz^{1/q})\Big),
    \]
    for any $\bz\in (0,\infty)^k$ (otherwise zero). In particular,
    \begin{equation}\label{powersfit}
    \log f_{\bZ}(\bz)\ \sim\ -\,\min_{\beps\in\{-1,1\}^k}\ \frac12\,(\beps\odot \bz^{1/q})^\top\bSigma^{-1}(\beps\odot \bz^{1/q})
    \quad (\|\bz_+\|\to\infty).
    \end{equation}
\end{lemma}

\noindent For \(k=2\) and \(\bSigma=\begin{pmatrix}1&\rho\\ \rho&1\end{pmatrix}\) with \(\rho\in(-1,1)\), Lemma \ref{lem:power-gaussian-density} gives,
\begin{align*}
f_{\bZ}(z_1,z_2)
&=\frac{2}{\pi q^2\sqrt{1-\rho^2}}\,z_1^{1/q-1}z_2^{1/q-1}
\exp\Big(-\frac{z_1^{2/q}+z_2^{2/q}}{2(1-\rho^2)}\Big)\,
\cosh\Big(\frac{\rho\,z_1^{1/q}z_2^{1/q}}{1-\rho^2}\Big).
\end{align*}
Thus, there exist functions \(p_1,p_2>0\) as in Lemma \ref{lem:density-implies-SE} such that, with
\begin{equation}\label{eq:density-bounds-J}
\overline{\calJ}(\bz)=\min_{\beps\in\{-1,1\}^k}\ \frac12\,(\beps\odot \bz^{1/q})^\top\bSigma^{-1}(\beps\odot \bz^{1/q})
\end{equation}
we have 
$$
p_1(\bz)e^{-\overline{\calJ}(\bz)}\le f_{\bZ}(\bz)\le p_2(\bz)e^{-\overline{\calJ}(\bz)}
$$
for all large enough \(\|\bz_+\|\). 

For any symmetric positive semidefinite matrix $\bA$, we denote by $\bA^+$ its \emph{Moore-Penrose pseudoinverse}. In expressions of the form $\bx^\top \bA^+ \bx$, we adopt the convention that the value is $+\infty$ if $\bx \not\in \operatorname{range}(\bA)$.

\begin{theorem}\label{thm:abs-powers-gauss}
    Let $\bG_1,\bG_2,\ldots$ be i.i.d.\ $\mathcal N(\bfm 0,\bSigma)$ in $\bbR^k$, where $\bSigma$ is positive semidefinite. Fix $q>2$. Set
    \[
    M_q=\EE|G|^q=\frac{2^{q/2}\Gamma\big((q+1)/2\big)}{\sqrt\pi},
    \qquad G\sim\mathcal N(0,1).
    \]
    Define
    \[
    \bX_i:=\big(|G_{i,1}|^q-\bSigma_{11}^{q/2}M_q,\ldots,|G_{i,k}|^q-\bSigma_{kk}^{q/2}M_q\big)\in\bbR^k,\qquad i\ge 1.
    \]
    Then for every $\bt\in(0,\infty)^k$,
    \[
    \lim_{N\to\infty}\frac{1}{N^{2/q}}
    \log\PP\Big(\frac1N\sum_{i=1}^N\bX_i\ge \bt\Big) = -\calI_{\calJ}(\bt),
    \]
    where\[
    \calI_{\calJ}(\bt)
    :=\min\Big\{\sum_{r=1}^k \inf_{\bs\ge \bt^{(r)}}\overline{\calJ}(\bs):
    \bt^{(r)}\in\bbR_+^k,\ \sum_{r=1}^k \bt^{(r)}=\bt\Big\},
    \]
    and
    \[
    \overline{\calJ}(\bz)
    :=\frac12\min_{\beps\in\{-1,1\}^k}
    (\beps\odot \bz^{1/q})^\top\bSigma^+(\beps\odot \bz^{1/q}),
    \qquad \bz\in\bbR_+^k.
    \]
\end{theorem}

\begin{corollary}\label{cor:rho-threshold}
    Let \(k=2\) with \(\bSigma=\begin{pmatrix}1&\rho\\ \rho&1\end{pmatrix}\), \(\rho\in(-1,1)\). Then, for any \(t_1,t_2>0\),
    \begin{align*}
    \lim_{N\to\infty}\frac{1}{N^{2/q}}
    &\log\PP \Big(\frac1N\sum_{i=1}^N\big(|G_{i,1}|^q-M_q,|G_{i,2}|^q-M_q\big)\ge (t_1,t_2)\Big)\\
    &= -\calI_{\calJ}(t_1,t_2),
    \end{align*}
    where
    \[
    \overline{\calJ}(z_1,z_2)
    :=\frac{z_1^{2/q}+z_2^{2/q}-2|\rho|\,z_1^{1/q}z_2^{1/q}}{2(1-\rho^2)},
    \qquad (z_1,z_2)\in\bbR_+^2.
    \]
\end{corollary}    
    
\noindent Interestingly, in this result, the rate function $\calI_{\calJ}$ is neither convex nor concave (see Lemma \ref{lem:I_not_convex_nor_concave_q_gt_2_rho_pos}).

\subsection{Geometry of high-dimensional \texorpdfstring{$\ell_p^N$}{lpN}-balls}
This section uses our analysis of multivariate stretched-exponential vectors to study the behavior of high-dimensional objects. We will see that an analogue to light- and heavy-tailed distributions exists for high-dimensional objects. A concrete class of high-dimensional objects is the family of \(\ell_p^N\)-balls, the unit balls in \(\bbR^N\) for the usual \(\ell_p\)-norm. This norm is given by
\[
\|\bx\|_p =
\begin{cases}
  \left(\sum_{i=1}^N |x_i|^p \right)^{1/p}, & \text{if } p \in [1, \infty), \\
  \max\{|x_1|, \ldots, |x_N|\}, & \text{if } p = \infty,
\end{cases}
\]
and $\BB_p^N = \{ \bx \in \bbR^N : \|\bx\|_p \le 1 \}$. Even though this class is only parametrized by the value of $p$, it gives rise to interesting and unexpected behavior in high dimensions and is an active area of research (see for instance \cite{AlonsoGutierrezProchnoThaele2018, GantertKimRamanan2017, KabluchkoProchnoThaele2019, KabluchkoProchnoThaele2021,kabluchko2024strange,kim2018conditional, prochno2024limit} or \cite{prochno2024large}).

For a matrix \(\bM\) write \(\bM^\top\) for the transpose. We denote by $\operatorname{Id}_m$ the $m$-dimensional identity matrix. For \(N\ge m\) the \emph{Stiefel manifold} is defined as
\[
\bbV_{m,N}:=\Big\{\bM\in\bbR^{m\times N}:\bM\bM^\top=\operatorname{Id}_m\Big\}
\]
and carries a unique Haar probability measure, denoted \(\operatorname{Unif}(\bbV_{m,N})\), invariant under left and right multiplication by orthogonal matrices. Elements of \(\bbV_{m,N}\) correspond, up to an orthogonal transformation, to orthogonal projections. More precisely, for any convex body \(C\subseteq\bbR^N\) and $\bV\in\bbV_{m,N}$,
\[
\bV C=\bV\big\vert_E\,(C\big| E),
\]
where \(E=\operatorname{Range}(\bV^\top)\) and \(C\big| E\) is the orthogonal projection of \(C\) onto \(E\). If \(\bx\in E\), then \(\bx=\sum_{i=1}^m\lambda_i\bv_i\) for some scalars \(\lambda_i\), where \(\{\bv_i\}_{i=1}^m\) are the orthonormal rows of \(\bV\). Hence, \(\bV\big\vert_E\,\bx=(\lambda_1,\ldots,\lambda_m)\), so \(\bV\big\vert_E\) maps the orthonormal basis \(\{\bv_i\}_{i=1}^m\) to the standard basis of \(\bbR^m\). In particular, for any convex body \(C\subseteq\bbR^N\),
\[
\operatorname{vol}_m(\bV C)\ \overset{d}{=}\ \operatorname{vol}_m(C\big| E),
\]
with \(\bV\sim \operatorname{Unif}(\bbV_{m,N})\) and \(E\) a uniformly chosen \(m\)-dimensional subspace of \(\bbR^N\), i.e. according to the Haar measure. The \emph{support function} of a convex body \(C\subseteq\bbR^m\) is
\[
h(C,\bu):=\sup_{\bx\in C}\langle \bx,\bu\rangle,\qquad \bu\in\bbS^{m-1},
\]
and determines \(C\) uniquely. We write \(q\) for the H\"older conjugate of \(p\), such that \(1/p+1/q=1\). It is known that the support function $h(N^{1/p-1/2}\bV_N\bbB_p^N,\cdot)$ converges almost surely in the uniform norm to the constant $\sqrt{2}\,\Gamma\big(\frac{q+1}{2}\big)^{1/q}\pi^{-1/(2q)}$ for fixed $m$ as $N\to\infty$ which is the support function of the Euclidean ball of radius $\sqrt{2}\,\Gamma\big(\frac{q+1}{2}\big)^{1/q}\pi^{-1/(2q)}$ (see \cite{prochno2024limit}). We will quantify this convergence and give rate function as well as the speed of the convergence (which is of different nature for $p\in(1,2)$ and $p\in(2,\infty]$) in the following. We show (see below) that for every
\[
f\in C \Big( \bbS^{m-1},\ \Big[\frac{\sqrt{2}\,\Gamma\big(\frac{q+1}{2}\big)^{1/q}}{\pi^{1/(2q)}},\,\infty\Big)\Big),
\]
one has for \(p\in(1,2)\),
\[
\PP\big(h(N^{1/p-1/2}\bV_N\bbB_p^N,\cdot)\ge f\big)\approx \exp\big(-N^{2/q}\,\calJ_p(f)\big),
\]
and for \(p\in(2,\infty]\),
\[
\PP\big(h(N^{1/p-1/2}\bV_N\bbB_p^N,\cdot)\ge f\big)\approx\exp\big(-N\,\calI_p(f)\big),
\]
for certain rate functions \(\calJ_p\) and \(\calI_p\). Here \(\calI_p\) arises from transformations of the Legendre--Fenchel transform in Cram\'er's theorem, while \(\calJ_p\) is a transformed quadratic form which originates from Theorem \ref{thm:LDP-multi}. Since support functions determine convex bodies, these asymptotics encode all geometric information. Thus, for large \(N\), the speed at which the rescaled projections \(N^{1/p-1/2}\bV_N\bbB_p^N\) approach the Euclidean ball of radius \(\sqrt{2}\,\Gamma\big(\frac{q+1}{2}\big)^{1/q}\pi^{-1/(2q)}\) is of order \( \exp(-c\,N^{2/q})\) for \(p\in(1,2)\) with explicit \(c>0\), and depends on the shape of \(\bbB_p^N\). As \(p\) is closer to $1$ the speed decreases, while for \(p\) being closer to $2$ it increases. In contrast, for \(p\in(2,\infty]\) the speed is always of order \(\exp(-c\,N)\) with a constant \(c>0\) where the speed does not depend on the shape of \(\bbB_p^N\).

We begin with an extension of the result for absolute powers of multivariate Gaussians. On the large deviation scale at speed $N^{2/q}$, the columns of a Haar-Stiefel matrix behave like standard Gaussian vectors.

\begin{theorem}\label{thm: LDP for finite stiefel evals}
Let \(\bV_N\sim \operatorname{Unif}(\bbV_{m,N})\) and let \(\bv_1,\ldots,\bv_N\) be the columns of \(\bV_N\). Fix \(k\in\bbN\), \(q>2\), \(t_1,\ldots,t_k>0\), and pairwise distinct \(\bu_1,\ldots,\bu_k\in\bbS^{m-1}\). Then
\begin{align*}
&\lim_{N\to\infty}\frac{1}{N^{2/q}}\log \PP\Big(
\frac{1}{N}\sum_{i=1}^N |\langle \sqrt{N}\,\bv_i,\bu_1\rangle|^q - M_q \ge t_1,\ \ldots,\
\frac{1}{N}\sum_{i=1}^N |\langle \sqrt{N}\,\bv_i,\bu_k\rangle|^q - M_q \ge t_k\Big)\\
& = -\calI_{\calJ}(t_1,\ldots,t_k),
\end{align*}
where $\langle \cdot, \cdot \rangle$ denotes the Euclidean inner product,  \(\calI_{\calJ}\) is the rate from Theorem~\ref{thm:abs-powers-gauss} with \(\bSigma = (\langle \bu_i,\bu_j\rangle)_{1\le i,j\le k}\), \(M_q=\EE|G|^q=\frac{2^{q/2}\Gamma(\frac{q+1}{2})}{\sqrt{\pi}}\) for \(G\sim\mathcal{N}(0,1)\).
\end{theorem}

\noindent The next result gives the large deviation behavior for projected \(\ell_p^N\)-balls for $p\in(1,2)$.

\begin{theorem}\label{thm: supprt function LDP p<2}
Let \(p\in(1,2)\) and \(\bV_N\sim \operatorname{Unif}(\bbV_{m,N})\) for all \(N\ge m\). Let \(q>2\) with \(1/p+1/q=1\), and let \(\{\bu_1,\bu_2,\ldots\}\subseteq\bbS^{m-1}\) be a dense countable set. For every \(f\in C(\bbS^{m-1},[M_q^{1/q},\infty))\),
\begin{align*}
&\lim_{N\to\infty}\frac{1}{N^{2/q}}\log \PP\big(h(N^{1/p-1/2}\bV_N\bbB_p^N,\cdot)\ge f\big) = -\sup_{k\in\bbN} \calJ_k\big(f(\bu_1),\ldots,f(\bu_k)\big),
\end{align*}
where for each $k\in\bbN$ we set
\[
\calJ_k(\bs)
:=\calI_{\calJ^{(k)}}\big((\bs^q-M_q\bfm1)\big),
\qquad \bs\in[M_q^{1/q},\infty)^k,
\]
and $\calI_{\calJ^{(k)}}$ is the rate from Theorem~\ref{thm:abs-powers-gauss} with covariance matrix
$\bSigma^{(k)}:=(\langle \bu_r,\bu_s\rangle)_{1\le r,s\le k}$.
\end{theorem}

The rate in Theorem~\ref{thm: supprt function LDP p<2} equals \(0\) for the constant function \(f\equiv M_q^{1/q}\) and is strictly positive otherwise. By \cite[Theorem 2.4]{prochno2024limit} we have the following: let \(p\in(2,\infty]\) and \(\bV_N\sim \operatorname{Unif}(\bbV_{m,N})\). Let \(q\in[1,2)\) with \(1/p+1/q=1\). Then, $(h(N^{1/p-1/2}\bV_N\bbB_p^N,\cdot))_{N\ge 1}$ satisfies an LDP on $C(\bbS^{m-1})$ with speed $N$ and good rate function $\calI_p$. One representation is
\[
\calI_p(f)= \inf\Big\{J(\nu): f(\bu)=\Big(\int_{\bbR^m}|\langle \bx,\bu\rangle|^q\,\dif\nu(\bx)\Big)^{1/q}\ \text{for all } \bu\in\bbS^{m-1}\Big\},
\]
where \(J\) is a good convex rate function on probability measures on \(\bbR^m\) (see  \cite[Theorem 2.4]{prochno2024limit}; see also \cite[Theorem 1]{kabluchko2024strange}). 

\subsection{Further applications and examples}\label{subsec:further-examples}

The class of multivariate stretched-exponentials in Definition \ref{def:multiSE} contains several standard models from
reliability, signal processing, extreme-value theory, and high-dimensional statistics. In this section, we give some examples that can be applied to our theory. In each example below, after centering, one checks Definition \ref{def:multiSE} (typically via Lemma \ref{lem:density-implies-SE} or by direct evaluation of upper-orthant tails) together with Assumption \ref{assu-LT}.

\begin{itemize}

\item \textbf{(Multivariate Weibull tails).}
Let $\alpha\in (0,1)$ and $\lambda_0,\lambda_1,\dots,\lambda_k>0$. Following \cite{hanagal1996multivariate}, a canonical type of multivariate Weibull distribution is specified by the survival function
\[
\PP(\bX>\bt)=\exp\Big(
-\sum_{j=1}^k \lambda_j t_j^\alpha
-\lambda_0 \max_{1\le j\le k} t_j^\alpha
\Big),
\qquad \bt\in\bbR_+^k.
\]
Such distributions arise naturally in competing-risks, reliability shock models and have a long
history in dependence modelling through the Marshall--Olkin scheme \cite{JoseRisticJoseph2011,MarshallOlkin1967}.

\item \textbf{(Multivariate generalized Gaussian).}
Fix $\alpha\in(0,1)$, $\bSigma\in\bbR^{k\times k}$ symmetric positive definite, and define the density by
\[
f_{\bX}(\bx)
=
c_{\alpha,k}|\bSigma|^{-1/2}\exp\Big(-\big(\bx^{\top}\bSigma^{-1}\bx\big)^{\alpha/2}\Big),
\qquad \bx\in\bbR^k,
\]
where $c_{\alpha,k}$ is a normalization constant. This class is often called multivariate generalized Gaussian distribution and is widely used in model-based clustering and as a parametric model for
wavelet statistics in image processing \cite{DangBrowneMcNicholas2015, GomezGomezVillegasMarin1998,VerdoolaegeScheunders2012}.

\item \textbf{(Gaussian scale mixtures).}
Let $\bZ\sim \mathcal{N}(0,\bSigma)$ and $R\ge0$ independent of $\bZ$, and set $\bX=R\bZ$.
If $R$ has Weibull-type tail $\log\PP(R>r)\sim -c r^\alpha$ with $\alpha\in(0,1)$, then $\bX$ inherits
stretched-exponential right tails (with a rate $\calJ$ determined by the minimal radial scaling needed
to enter the orthant $x\bt$). Gaussian scale mixtures are a standard tool in sparse signal modelling
and appear in finance models \cite{HashorvaPakesTang2010,WainwrightSimoncelli1999}. A related case is covered in \cite{hagele2021large} for $\bZ$ being a random vector on the unit sphere.

\item \textbf{(Polynomial images of Gaussians).}
Let $\bG\sim\mathcal{N}(0, \Id_m)$ and define 
$$
\bX=(P_1(\bG),\dots,P_k(\bG)),
$$
where each $P_j$ is a non-degenerate homogeneous polynomial of total degree $d\ge3$. Tail estimates for Gaussian chaoses imply stretched-exponential tails of order $2/d\in(0,1)$
for such polynomial observables \cite{Latala2006}. These objects appear in stochastic
analysis (Wiener chaos expansions) and in random matrix theory.

\item \textbf{(Sub-Weibull vectors in high-dimensional covariance problems).}
In high-dimensional statistics one may assume coordinates with Weibull-type tails
$\PP(|X_j|>t)=\exp(-(t/K)^\alpha)$ for some $\alpha\in(0,1)$, in order to go beyond sub-Gaussianity (see for example \cite{KuchibhotlaChakrabortty2018}). Then centered quadratic and bilinear forms (e.g.,\ entries of sample covariance matrices) are sums of i.i.d.\ stretched-exponential terms, which is a key input for concentration and moderate/large deviation arguments in covariance estimation and regression under sub-Weibull tails.

\end{itemize}

\section{Preliminaries}\label{sec: preliminaries}
We introduce some notation and definitions.  For a random vector $\bX$ that is distributed according to a probability measure $\mu$, we write $\bX\sim \mu$. The mean of $\bX$ is denoted by $\EE[\bX]$ and the covariance by $\Cov[\bX]$. An open ball of radius $r > 0$ centered at $\bx$ in a metric space is denoted by $\BB_r(\bx)$. We use both $\bfm 1_{A}(\bx)$ and $\bfm 1\{\bx\in A\}$ for the indicator function over some set $A$. For a square matrix $\bA$, $\operatorname{Tr}(\bA)$ denotes the trace of $\bA$ and for a measurable set $B$, $\operatorname{vol}_m(B)$ is the $m$-dimensional Lebesgue measure of $B$. \\

Let $\XX$ be a topological space. A function $I\colon \XX \to [0,\infty]$ is a \emph{rate function} if all sublevel sets $\{x\in\XX: I(x)\le y\}$ are closed for every $y\in \RR$. A rate function is \emph{good} if all sublevel sets are compact.

\begin{definition}\label{def: LDP}
Let $(Z_N)_{N\in\bbN}$ be random elements in a topological space $\XX$. The sequence $(Z_N)_{N\in\bbN}$ satisfies a \emph{large deviation principle} with speed $s_N\to\infty$ and rate function $I\colon \XX\to[0,\infty]$ if
\begin{align}
\liminf_{N\to\infty}\frac{1}{s_N}\log \PP(Z_N\in O)
&\ge -\inf_{x\in O} I(x)\quad \text{for all open } O\subseteq \XX,\label{LDP lower bound}\\
\limsup_{N\to\infty}\frac{1}{s_N}\log \PP(Z_N\in C)
&\le -\inf_{x\in C} I(x)\quad \text{for all closed } C\subseteq \XX.\label{LDP upper bound}
\end{align}
\end{definition}

\begin{proposition}[Cram\'er's theorem,  \text{\cite[Chapter 24]{kallenberg1997foundations}}]\label{prop:cramer}
Let $\bX,\bX_1,\bX_2,\ldots$ be i.i.d. in $\RR^d$. Assume the \emph{log-moment generating function}
\[
\Lambda(\blambda):=\log\EE\big[e^{\langle \blambda, \bX\rangle}\big],\qquad \blambda\in\RR^d,
\]
is finite in a neighborhood of $\bfm 0\in\RR^d$. Then $\big(\frac{1}{N}\sum_{i=1}^N \bX_i\big)_{N\ge1}$ satisfies an LDP on $\RR^d$ with speed $N$ and good, convex rate function $\Lambda^*$, the \emph{Fenchel--Legendre transform} of $\Lambda$,
\[
\Lambda^*(\bx):=\sup_{\blambda\in\RR^d}\big(\langle \blambda,\bx\rangle-\Lambda(\blambda)\big),\qquad \bx\in\RR^d.
\]
\end{proposition}

\begin{proposition}[Contraction principle, \text{\cite[Theorem 4.2.1]{dembo2009techniques}}]\label{prop:contraction}
Let $\XX,\YY$ be Hausdorff topological spaces and $f\colon\XX\to\YY$ continuous. If $(Z_N)_{N\ge1}$ satisfies an LDP on $\XX$ with speed $s_N$ and good rate function $I\colon\XX\to[0,\infty]$, then $(f(Z_N))_{N\ge1}$ satisfies an LDP on $\YY$ with speed $s_N$ and good rate function $J\colon\YY\to[0,\infty]$ given by
\[
J(y)=\inf\{I(x):\ f(x)=y\},\qquad y\in\YY.
\]
By convention, $J(y)=+\infty$ if $f^{-1}(\{y\})=\varnothing$.
\end{proposition}

\noindent We state the moderate deviation principle in our notation. See \cite[Theorem 4.7]{gao2011delta}.

\begin{proposition}[Gaussian range for the empirical mean, general]\label{prop: MDP general}
Let $\XX$ be a measurable space and let $X$ take values in $\XX$ with law $\mu$. Let $\TT\subseteq\RR^d$ be a parameter set and let $\psi:\XX\times\TT\to\RR^d$ be measurable. Let $(a_N)_{N\in\bbN}$ be such that $a_N\to\infty$ and $a_N/\sqrt{N}\to 0$. For $\btheta\in\TT$ set
\[
\bPsi(\btheta):=\EE[\psi(X,\btheta)],\qquad
\bPsi_N(\btheta):=\frac{1}{N}\sum_{i=1}^N \psi(X_i,\btheta),
\]
where $X_1,\dots,X_N$ are i.i.d.\ with law $\mu$. Assume:

(C1) For each $x\in\XX$, the map $\btheta\mapsto\psi(x,\btheta)$ is continuous. For each $\btheta\in\TT$, the map $x\mapsto\psi(x,\btheta)$ is measurable.

(C2) There exist $\btheta_0\in\TT$ and $\eta>0$ such that:
\begin{itemize}[nosep,leftmargin=*]
\item $\bPsi(\btheta_0)=0$ and $\btheta_0$ is the unique zero of $\bPsi$ in $\TT$,
\item $\BB_\eta(\btheta_0)\subseteq \TT$,
\item $\bPsi$ is differentiable at $\btheta_0$ with nonsingular derivative $\bA\in\RR^{d\times d}$,
\item $\bPsi$ is a homeomorphism on $\BB_\eta(\btheta_0)$,
\item $\EE[\|\psi(X,\btheta_0)\|^2]<\infty$.
\end{itemize}

(C3) 
\[
\frac{\sqrt{N}}{a_N}\sup_{\btheta\in\BB_\eta(\btheta_0)}
\|\bPsi_N(\btheta)-\bPsi(\btheta)\|\xrightarrow{\PP}0,\,
\limsup_{N\to\infty}\frac{1}{a_N^2}\log\Big(N\,\PP\big(\sup_{\btheta\in\BB_\eta(\btheta_0)}\|\psi(X,\btheta)\|\ge \sqrt{N}\,a_N\big)\Big)=-\infty.
\]

Let $\bPsi_{0N}$ be the restriction of $\bPsi_N$ to $\BB_\eta(\btheta_0)$. Let $\Phi$ be the inverse map functional from \cite[Lemma 4.1]{gao2011delta} and set $\btheta_N:=\Phi(\bPsi_{0N})$. Then
\[
\frac{\sqrt{N}}{a_N}\,(\btheta_N-\btheta_0)
\]
satisfies an LDP with speed $a_N^2$ and good rate function
\[
I(\bz)=\frac12\langle \bA \bz,\bGamma^{-1}\bA \bz\rangle,
\]
where $\bGamma=\Cov(\psi(X,\btheta_0))$ is invertible.
\end{proposition}

\noindent We deduce the following simplified version of Proposition \ref{prop: MDP general}.

\begin{corollary}[Gaussian range in $\RR^k$]\label{cor: MDP in Rd}
Let $\bX, \bX_1,\dots, \bX_N$ be i.i.d.\ in $\RR^k$ with mean $\EE[\bX]$ and invertible covariance matrix 
$\bSigma=\Cov(\bX)$. Let $(a_N)_{N\in\bbN}$ be such that $a_N\to\infty$ and $a_N/\sqrt{N}\to 0$. Assume
\begin{align}
\limsup_{N\to\infty}\frac{1}{a_N^2}
\log\Big(N\,\PP\big(\|\bX\|\ge \sqrt{N}\,a_N\big)\Big)=-\infty.
\label{eq: cond for MDP}
\end{align}
Then
\[
\frac{1}{a_N\sqrt{N}}\sum_{i=1}^N (\bX_i-\EE[\bX])
\]
satisfies an LDP with speed $a_N^2$ and good rate function $I(\bz)=\frac12\langle \bz,\bSigma^{-1}\bz\rangle$.
\end{corollary}

Condition \eqref{eq: cond for MDP} is weaker than the existence of exponential moments and thus suited for stretched-exponential tails.

\section{Proofs}\label{sec: Proofs}
In this section, we provide all the proofs. We start with some preliminary lemmas. After that we give the proof of the main result and we end the section with proofs of the applications. 

\subsection{Preliminary lemmas}

We first show that the function $\calJ$  in Definition \ref{def:multiSE} is always $\alpha$-homogeneous.
\begin{lemma}\label{lem:homogeneity_barJ}
    If $\bY$ satisfies Definition \ref{def:multiSE} with rate $(\calJ,\alpha)$, then $\calJ$ is $\alpha$-homogeneous.
\end{lemma}
    
    \begin{proof}
    Fix $\bu\in\bbR_+^k\setminus\{\bfm 0\}$ and $\lambda>0$. By Definition \ref{def:multiSE}, we get
    \begin{align*}
    -\calJ(\lambda\bu)
    &=\lim_{x\to\infty}\frac{1}{x^\alpha}\log \PP(\bY\ge x\lambda\bu)
    =\lim_{y\to\infty}\frac{1}{(y/\lambda)^\alpha}\log \PP(\bY\ge y\bu) \\
    &=\lambda^\alpha \lim_{y\to\infty}\frac{1}{y^\alpha}\log \PP(\bY\ge y\bu)
    =-\lambda^\alpha \calJ(\bu),
    \end{align*}
    hence $\calJ(\lambda\bu)=\lambda^\alpha \calJ(\bu)$.
\end{proof}

As the following lemma shows, the random variables that have stretched-exponential tails in the sense of Definition \ref{def:multiSE} always have second moments.

\begin{lemma}[Marginal tails and finite second moments]
    \label{lem:finite-second}
    If $\bY$ satisfies \Cref{def:multiSE} and Assumption \ref{assu-LT}, then there exist constants $c,C>0$ such that for all $j \in \{1,\dots,k\}$, $\PP(|Y_j|>t)\le c e^{-C t^\alpha}$ for $t$ large enough and, in particular, $\EE[Y_j^2]<\infty$.
\end{lemma}

\begin{proof}[Proof of Lemma \ref{lem:finite-second}]
    For $t\to\infty$ and $\be_j$ the $j$-th unit vector, consider
    \[
    \bv^{(j)}(t):=(-t,\ldots,-t,t,-t,\ldots,-t)\in\bbR^k.
    \]
    Then $(\bv^{(j)}(t))_+=t\be_j$ and $\bv^{(j)}(t)=t\,\bv^{(j)}(1)$. Hence, by \Cref{def:multiSE},
    \[
        \lim_{t\to\infty}\frac{1}{t^\alpha}\log\PP(\bY\ge \bv^{(j)}(t))
        = -\calJ(\be_j).
    \]
    Set $d_j:=\calJ(\be_j)>0$. Therefore, for all large $t$,
    \[
    \PP(\bY\ge \bv^{(j)}(t)) \le \exp\Big(-\frac{d_j}{2}\,t^\alpha\Big).
    \]
    Moreover,
    \[
    \PP(Y_j\ge t)
    \le \PP\Big(Y_j\ge t,\ \forall \ell\ne j:\ Y_\ell\ge -t\Big) + \sum_{\ell\ne j}\PP(Y_\ell\le -t).
    \]
    The first term is bounded by $\PP(\bY\ge \bv^{(j)}(t))$, and Assumption~\ref{assu-LT} yields
    $\PP(Y_\ell\le -t)\le \exp(-c_\ell t^\alpha)$ for all large $t$. Since $k<\infty$, there exist $t_0\ge 1$ and
    constants $C_+,C_->0$ such that for all $t\ge t_0$ and all $j$,
    \[
    \PP(Y_j\ge t)\le e^{-C_+ t^\alpha},
    \qquad
    \PP(Y_j\le -t)\le e^{-C_- t^\alpha}.
    \]
    Hence for $t\ge t_0$ and all $j$, $\PP(|Y_j|>t)\le 2e^{-C t^\alpha}$ with $C:=\min\{C_+,C_-\}$, which implies
    $\EE[Y_j^2]<\infty$.
\end{proof}

\begin{lemma}\label{lem:Ij_k_summands}
    Let $\calJ: \bbR^k_+\to[0,\infty)$ be as in Definition \ref{def:multiSE} and recall the definition of  $\calI_\calJ$ in \eqref{def-rate-fct}. Then, for every
    $\bt\in\bbR^k_+$,
    \begin{equation}\label{rate-fct-second}
    \calI_\calJ(\bt) 
        = \min_{m\in\NN} \inf\Big\{\sum_{r=1}^m \calJ\big(\bt^{(r)}\big): \bt^{(r)}\in\bbR^k_+,\ \sum_{r=1}^m \bt^{(r)}\ge \bt\Big\} 
       \end{equation}
    Further, $\calI_\calJ$ is nondecreasing in each component.
\end{lemma}
    
    \begin{proof}
    First, we note that $\calJ$ is lower semicontinuous, componentwise nondecreasing, and $\alpha$-homogeneous (for some $\alpha\in(0,1)$, see Lemma \ref{lem:homogeneity_barJ}). Given $\bt^{(1)},\dots,\bt^{(m)}\in\bbR^k_+$ with $\sum_{r=1}^m \bt^{(r)}\ge \bt$, define recursively
    $\bl^{(1)}:=\bt$ and
    \[
    \bu^{(r)}:=\min\{\bt^{(r)},\bl^{(r)}\},\qquad \bl^{(r+1)}:=\bl^{(r)}-\bu^{(r)},\qquad r=1,\dots,m,
    \]
    where minima are componentwise. Then $\bu^{(r)}\le \bt^{(r)}$ and $\sum_{r=1}^m \bu^{(r)}=\bt$ by construction (using an induction argument), hence
    $\sum_{r=1}^m \calJ(\bu^{(r)})\le \sum_{r=1}^m \calJ(\bt^{(r)})$.
    Therefore we can assume equality in the r.h.s. of \eqref{rate-fct-second}.    
    Now, fix $m\in\NN$ and $\bt^{(1)},\dots,\bt^{(m)}\in\bbR^k_+$ with $\sum_{r=1}^m \bt^{(r)}=\bt$.
Let $\hat a_r:=\|\bt^{(r)}\|_1$ and set $\be_1:=(1,0,\dots,0)\in\bbR^k$. Define
$\bu^{(r)}:=\bt^{(r)}/\hat a_r$ if $\hat a_r>0$, and $\bu^{(r)}:=\be_1$ if $\hat a_r=0$.
Then $\|\bu^{(r)}\|_1=1$ and $\bt^{(r)}=\hat a_r\bu^{(r)}$ for all $r$, hence by $\alpha$-homogeneity,
$\sum_{r=1}^m \calJ(\bt^{(r)})=\sum_{r=1}^m \hat a_r^\alpha\,\calJ(\bu^{(r)})$. Consider the convex polytope
\[
P:=\Big\{\ba=(a_1,\dots,a_m)\in[0,\infty)^m:\ \sum_{r=1}^m a_r\bu^{(r)}=\bt\Big\},
\]
and $F:P\to[0,\infty)$ given by $F(\ba):=\sum_{r=1}^m a_r^\alpha\,\calJ(\bu^{(r)})$.
Note that $\hat{\ba}:=(\hat a_1,\dots,\hat a_m)\in P$ and $F(\hat{\ba})=\sum_{r=1}^m \calJ(\bt^{(r)})$.
The function $F$ is concave, hence there exists an extreme point $\ba^\ast\in P$ with $F(\ba^\ast)\le F(\hat{\ba})$.
If $\ba^\ast$ has more than $k$ positive components, then the corresponding vectors $\bu^{(r)}\in\bbR^k$
(with $a_r^\ast>0$) are linearly dependent, so there exists a nonzero $\bd=(d_1,\dots,d_m)$ with
$d_r=0$ whenever $a_r^\ast=0$ and $\sum_{r=1}^m d_r\bu^{(r)}=0$. For any $\varepsilon>0$ such that
$\ba^\ast\pm\varepsilon\bd\in[0,\infty)^m$, we have
$\ba^\ast\pm\varepsilon\bd\in P$, contradicting extremality of $\ba^\ast$. Hence $\ba^\ast$ has at most $k$
positive components. Writing these indices as $r_1,\dots,r_\ell$ with $\ell\le k$ and setting
$\tilde\bt^{(j)}:=a_{r_j}^\ast\,\bu^{(r_j)}$, we obtain $\bt=\sum_{j=1}^\ell \tilde\bt^{(j)}$ and
\[
\sum_{j=1}^\ell \calJ(\tilde\bt^{(j)})=\sum_{j=1}^\ell (a_{r_j}^\ast)^\alpha\,\calJ(\bu^{(r_j)})
=F(\ba^\ast)\le F(\hat{\ba})=\sum_{r=1}^m \calJ(\bt^{(r)}).
\]
Since $\calJ$ is lower semicontinuous, it attains its infimum on the compact set. The fact that $\calI_\calJ$ is nondecreasing follows from the fact that $\calJ$ is nondecreasing and by \eqref{def-rate-fct}.
\end{proof} 

\subsection{Proof of Theorem \ref{thm:LDP-multi}}

In this section, we provide the proof of the main result, starting with the lower bound. 

\begin{lemma}[Lower bound of Theorem \ref{thm:LDP-multi}]\label{lem:lower-bound}
 In the setup of Theorem \ref{thm:LDP-multi},
 \[
\liminf_{N\to\infty}\frac{1}{x_N^\alpha}\log \PP\left(\frac{1}{x_N}\sum_{i=1}^N \bX_i \ge \bt\right)
\geq -\calI_{\calJ}(\bt).
\]
\end{lemma}

\begin{proof}
    Write $\bS_N:=\sum_{i=1}^N\bX_i$.
    Fix $\delta\in(0,1)$.
    We prove
    \begin{equation}\label{eq:lb-delta}
    \liminf_{N\to\infty}\frac{1}{x_N^\alpha}\log \PP(\bS_N\ge x_N\bt)
    \ge -\calI_\calJ((1+\delta)\bt),
    \end{equation}
    and then let $\delta\downarrow0$ using the $\alpha$-homogeneity of $\calI_\calJ$. For every fixed $\bt\in\RR_+^k$,
    \begin{equation}\label{eq:lb-one-point}
    \lim_{x\to\infty}\frac{1}{x^\alpha}\log \PP(\bX_1\ge x\bt)= -\calJ(\bt).
    \end{equation}
    Fix $\varepsilon>0$. By Lemma \ref{lem:Ij_k_summands}, there exist $m\in\NN$ and vectors
    $\bu^{(1)},\dots,\bu^{(m)}\in\RR_+^k$ such that
    \begin{equation}\label{eq:lb-choice}
    \sum_{r=1}^m \bu^{(r)}\ge (1+\delta)\bt,
    \qquad
    \sum_{r=1}^m \calJ(\bu^{(r)})\le \calI_\calJ((1+\delta)\bt)+\varepsilon.
    \end{equation}
    Define the events
    \[
    A_N:=\bigcap_{r=1}^m \{\bX_r\ge x_N\bu^{(r)}\},
    \qquad
    B_N:=\Big\{\sum_{i=m+1}^N \bX_i\ge -\delta x_N\bt\Big\}.
    \]
    On $A_N\cap B_N$,
    \[
    \bS_N=\sum_{r=1}^m \bX_r+\sum_{i=m+1}^N \bX_i
    \ge x_N\sum_{r=1}^m \bu^{(r)}-\delta x_N\bt
    \ge x_N\bt
    \]
    by \eqref{eq:lb-choice}. Hence,
    \begin{equation}\label{eq:lb-lower-event}
    \PP(\bS_N\ge x_N\bt)\ge \PP(A_N\cap B_N)=\PP(A_N)\PP(B_N),
    \end{equation}
    where the last equality uses independence. By \eqref{eq:lb-one-point},
    \[
    \lim_{N\to\infty}\frac{1}{x_N^\alpha}\log \PP(\bX_1\ge x_N\bu^{(r)})=-\calJ(\bu^{(r)}),
    \qquad r=1,\dots,m.
    \]
    Therefore,
    \begin{equation}\label{eq:lb-A}
    \lim_{N\to\infty}\frac{1}{x_N^\alpha}\log \PP(A_N)
    =
    \sum_{r=1}^m \lim_{N\to\infty}\frac{1}{x_N^\alpha}\log \PP(\bX_1\ge x_N\bu^{(r)})
    =
    -\sum_{r=1}^m \calJ(\bu^{(r)}).
    \end{equation}
    By Slutsky's theorem and the central limit theorem, $\PP(B_N) \to 1$, and so
    \begin{equation}\label{eq:lb-B}
    \liminf_{N\to\infty}\frac{1}{x_N^\alpha}\log \PP(B_N)=0,
    \end{equation}
    since the random variables are centered, have finite second moments (see Lemma \ref{lem:finite-second}) and $x_N /\sqrt{N} \to\infty$.
Combining \eqref{eq:lb-lower-event}, \eqref{eq:lb-A} and \eqref{eq:lb-B},
    \[
    \liminf_{N\to\infty}\frac{1}{x_N^\alpha}\log \PP(\bS_N\ge x_N\bt)
    \ge
    -\sum_{r=1}^m \calJ(\bu^{(r)}).
    \]
    Using \eqref{eq:lb-choice} and then letting $\varepsilon\downarrow0$ yields \eqref{eq:lb-delta}. Finally, by $\alpha$-homogeneity of $\calI_\calJ$, letting $\delta\downarrow0$ and taking the supremum over $m\in\NN$ and $\bu^{(1)},\dots,\bu^{(m)}\in\RR_+^k$ such that $\sum_{r=1}^m \bu^{(r)}\ge (1+\delta)\bt$ gives the claimed lower bound (see Lemma \ref{lem:Ij_k_summands}).
    \end{proof}

We continue with the proof of the upper bound and start with a preliminary lemma.

\begin{lemma}\label{lem:upperbound-small-part}
    Fix $\alpha\in(0,1)$ and $\bt\in(0,\infty)^k$. Let $(\bX_i)_{i\ge1}$ be i.i.d.\ centered random vectors with values in $\bbR^k$ and assume that $\bX_1$ satisfies \Cref{def:multiSE} with rate $(\calJ,\alpha)$ and Assumption \ref{assu-LT}. Let $(x_N)_{N\ge1}$ be a sequence of positive numbers such that $x_N\to\infty$ and $x_N N^{-1/(2-\alpha)}\to\infty$. Then
    \[
    \limsup_{\varepsilon_0\downarrow 0}\ \limsup_{N\to\infty}\frac{1}{x_N^\alpha}
    \log \PP\Big(\sum_{i=1}^N \bX_i \ge x_N\bt,\ \forall i\le N:\ \bX_i \le \varepsilon_0 x_N \bt\Big)
    =-\infty.
    \]
\end{lemma}
    
\begin{proof}
    Fix $\varepsilon_0\in(0,1)$ and set
    \[
    A_N(\varepsilon_0)
    :=\Big\{\sum_{i=1}^N\bX_i\ge x_N\bt,\ \forall i\le N:\ \bX_i\le \varepsilon_0 x_N\bt\Big\}.
    \]
    Then $A_N(\varepsilon_0)\subseteq \bigcap_{j=1}^k A_{N,j}(\varepsilon_0)$, where
    \[
    A_{N,j}(\varepsilon_0)
    :=\Big\{\sum_{i=1}^N X_{i,j}\ge x_N t_j,\ \forall i\le N:\ X_{i,j}\le \varepsilon_0 x_N t_j\Big\}.
    \]
    Hence, $\PP(A_N(\varepsilon_0))\le \min_{1\le j\le k}\PP(A_{N,j}(\varepsilon_0))$.
    It suffices to show that for each fixed $j$ there exists $c_*=c_*(j)>0$ such that
    \begin{equation}\label{eq:1d-target-correct}
    \limsup_{N\to\infty}\frac{1}{x_N^\alpha}\log \PP(A_{N,j}(\varepsilon_0))
    \le -\frac{c_*}{\varepsilon_0^{1-\alpha}}.
    \end{equation}
    Letting $\varepsilon_0\downarrow 0$ then yields the claim. Fix $j$ and write $\xi_i:=X_{i,j}$ and $t:=t_j$. Set $b_N:=\varepsilon_0 x_N t$, so that
    \[
    \PP(A_{N,j}(\varepsilon_0))
    =\PP\Big(\sum_{i=1}^N \xi_i\ge x_N t,\ \forall i\le N:\ \xi_i\le b_N\Big).
    \]
    By Lemma~\ref{lem:finite-second}, there exist constants $c,C>0$ such that
    \begin{equation}\label{eq:tail-1d}
    \PP(\xi_1>x)\le c\,e^{-C x^\alpha},\qquad x\ge 0.
    \end{equation}
    Fix $\eta\in(0,C)$. By Markov's inequality,
    \begin{align}
    \PP(A_{N,j}(\varepsilon_0))
    &=\PP\Big(\exp\Big(\eta b_N^{\alpha-1}\sum_{i=1}^N\xi_i\Big)\ge \exp\big(\eta b_N^{\alpha-1}x_N t\big),\ \forall i\le N:\ \xi_i\le b_N\Big)\notag\\
    &\le \exp\big(-\eta b_N^{\alpha-1}x_N t\big)
    \Big(\EE\big[\exp(\eta b_N^{\alpha-1}\xi_1)\mathbf 1_{\{\xi_1\le b_N\}}\big]\Big)^N.
    \label{eq:chernoff-correct}
    \end{align}
    Moreover,
    \begin{equation}\label{eq:theta-mainterm-correct}
    \eta b_N^{\alpha-1}x_N t
    =\eta(\varepsilon_0 x_N t)^{\alpha-1}x_N t
    =\frac{\eta t^\alpha}{\varepsilon_0^{1-\alpha}}\,x_N^\alpha.
    \end{equation}
    
    We bound the truncated moment generating function. Using the identity
 \[
    e^{y}
    =
    1+y
    +y^2\int_0^1(1-s)e^{sy}\,ds,
    \]
with $y = \eta b_N^{\alpha-1}\xi_1$, 
   multiplying by $\mathbf 1_{\{\xi_1\le b_N\}}$ and taking expectations gives
    \begin{align}
    \EE\big[e^{\eta b_N^{\alpha-1}\xi_1}\mathbf 1_{\{\xi_1\le b_N\}}\big]
    &= \EE\big[\mathbf 1_{\{\xi_1\le b_N\}}\big]
    +\eta b_N^{\alpha-1}\EE\big[\xi_1\mathbf 1_{\{\xi_1\le b_N\}}\big]\notag\\
    &\quad
    +\eta^2 b_N^{2\alpha-2}\EE\Big[\xi_1^2\mathbf 1_{\{\xi_1\le b_N\}}
    \int_0^1(1-s)e^{s\eta b_N^{\alpha-1}\xi_1}\,ds\Big].
    \label{eq:mgf-identity}
    \end{align}
    Since $\EE[\xi_1]=0$,
    \[
    \EE\big[\xi_1\mathbf 1_{\{\xi_1\le b_N\}}\big]
    =-\EE\big[\xi_1\mathbf 1_{\{\xi_1>b_N\}}\big]\le 0,
    \]
    hence
    \begin{equation}\label{eq:mgf-upper-reduced}
    \EE\big[e^{\eta b_N^{\alpha-1}\xi_1}\mathbf 1_{\{\xi_1\le b_N\}}\big]
    \le 1+\eta^2 b_N^{2\alpha-2}\EE\Big[\xi_1^2\mathbf 1_{\{\xi_1\le b_N\}}
    \int_0^1(1-s)e^{s\eta b_N^{\alpha-1}\xi_1}\,ds\Big].
    \end{equation}
    For $\xi_1\le 0$ we have $e^{s\eta b_N^{\alpha-1}\xi_1}\le 1$, hence
    \[
    \EE\Big[\xi_1^2\mathbf 1_{\{\xi_1\le 0\}}\int_0^1(1-s)e^{s\eta b_N^{\alpha-1}\xi_1}\,ds\Big]
    \le \EE[\xi_1^2]<\infty.
    \]
    For $0<\xi_1\le b_N$, fix $s\in[0,1]$ and use integration by parts for Stieltjes integrals to obtain
    \begin{equation}\label{eq:ibp}
    \EE\big[\xi_1^2 e^{s\eta b_N^{\alpha-1}\xi_1}\mathbf 1_{\{0<\xi_1\le b_N\}}\big]
    \le \int_0^{b_N} \big(2x+s\eta b_N^{\alpha-1}x^2\big)e^{s\eta b_N^{\alpha-1}x}\,\PP(\xi_1>x)\,dx.
    \end{equation}
    Since $\alpha-1<0$ and $0<x\le b_N$, we have $b_N^{\alpha-1}\le x^{\alpha-1}$, hence
    \[
    s\eta b_N^{\alpha-1}x \le \eta x^\alpha,
    \qquad
    s\eta b_N^{\alpha-1}x^2 \le \eta x^{\alpha+1}.
    \]
    Using \eqref{eq:tail-1d} in \eqref{eq:ibp} yields
    \[
    \EE\big[\xi_1^2 e^{s\eta b_N^{\alpha-1}\xi_1}\mathbf 1_{\{0<\xi_1\le b_N\}}\big]
    \le c\int_0^\infty (2x+\eta x^{\alpha+1})e^{-(C-\eta)x^\alpha}\,dx
    =:C_1<\infty,
    \]
    uniformly in $N$ and $s\in[0,1]$. Therefore, for some $C_2<\infty$ and all large $N$,
    \[
    \EE\Big[\xi_1^2\mathbf 1_{\{\xi_1\le b_N\}}\int_0^1(1-s)e^{s\eta b_N^{\alpha-1}\xi_1}\,ds\Big]\le C_2,
    \]
    and \eqref{eq:mgf-upper-reduced} gives
    \[
    \EE\big[e^{\eta b_N^{\alpha-1}\xi_1}\mathbf 1_{\{\xi_1\le b_N\}}\big]\le 1+C_2\eta^2 b_N^{2\alpha-2},
    \qquad
    \log \EE\big[e^{\eta b_N^{\alpha-1}\xi_1}\mathbf 1_{\{\xi_1\le b_N\}}\big]\le C_2\eta^2 b_N^{2\alpha-2},
    \]
    for all large $N$. Combining with \eqref{eq:chernoff-correct} and \eqref{eq:theta-mainterm-correct} yields
    \[
    \frac{1}{x_N^\alpha}\log \PP(A_{N,j}(\varepsilon_0))
    \le -\frac{\eta t^\alpha}{\varepsilon_0^{1-\alpha}} + C_2\eta^2\frac{N b_N^{2\alpha-2}}{x_N^\alpha}.
    \]
    Since $b_N^{2\alpha-2}=(\varepsilon_0 t)^{2\alpha-2}x_N^{2\alpha-2}$,
    \[
    \frac{N b_N^{2\alpha-2}}{x_N^\alpha}
    =(\varepsilon_0 t)^{2\alpha-2}\frac{N}{x_N^{2-\alpha}}
    \longrightarrow 0
    \qquad (N\to\infty),
    \]
    because $x_N N^{-1/(2-\alpha)}\to\infty$. Thus \eqref{eq:1d-target-correct} holds with $c_*=\eta t_j^\alpha$.
    Letting $\varepsilon_0\downarrow 0$ completes the proof.
    \end{proof}        

\begin{lemma}[Upper bound of Theorem \ref{thm:LDP-multi}]\label{lem:upper-bound-multi}
In the setup of Theorem \ref{thm:LDP-multi},
\[
\limsup_{N\to\infty}\frac{1}{x_N^\alpha}\log \PP\left(\frac{1}{x_N}\sum_{i=1}^N \bX_i \ge \bt\right)
\le -\calI_{\calJ}(\bt).
\]
\end{lemma}

\begin{proof}
    Fix $\bt=(t_1,\ldots,t_k)\in(0,\infty)^k$ and recall the rate function $\calI_\calJ$ from \eqref{def-rate-fct} and  \eqref{rate-fct-second}. Write
    \[
    \PP\Big(\frac{1}{x_N}\sum_{i=1}^N \bX_i\ge \bt\Big)
    =\PP\Big(\sum_{i=1}^N \bX_i\ge x_N\bt\Big).
    \]
    By Lemma~\ref{lem:finite-second} there exist $c>0$ and, for each $j$, some $u_j>0$ such that for all $u\ge u_j$,
    $\PP(|X_{1,j}|\ge u)\le 2e^{-c u^\alpha}$.
    Moreover, $x_N N^{-1/(2-\alpha)}\to\infty$ implies $\log(Nk)=o(x_N^\alpha)$.
    Hence we can choose numbers $\lambda_j>1$ such that for all large $N$,
    \begin{equation}\label{box-priori}
    Nk\max_{1\le j\le k}\PP\Big(\frac{|X_{1,j}|}{t_j}\ge \lambda_j x_N\Big)
    \le e^{-\calI_{\calJ}(\bt)\, x_N^\alpha}.
    \end{equation}
    Let $\lambda:=\max_{1\le j\le k}\lambda_j$ and define the event
    \[
    \calE_{\mathrm{box}}
    :=\Big\{\forall\,i\le N,\ \forall\,j\le k:\ \frac{X_{i,j}}{t_j}\in[-\lambda x_N,\lambda x_N]\Big\}.
    \]
    Then
    \begin{equation}\label{splitwithbox}
    \PP\Big(\frac{1}{x_N}\sum_{i=1}^N \bX_i\ge \bt\Big)
    \le
    \PP\Big(\sum_{i=1}^N \bX_i\ge x_N\bt,\ \calE_{\mathrm{box}}\Big)
    +\PP(\calE_{\mathrm{box}}^{\mathrm c}).
    \end{equation}
    By \eqref{box-priori} and a union bound,
    \begin{equation}\label{eq:eboxc-rate}
    \limsup_{N\to\infty}\frac{1}{x_N^\alpha}\log \PP(\calE_{\mathrm{box}}^{\mathrm c})
    \le -\calI_{\calJ}(\bt).
    \end{equation}
    Hence, it remains to bound the first term in \eqref{splitwithbox}.
    Fix $\delta\in(0,1/4)$ and $\varepsilon_0\in(0,1)$, and set
    \[
    \calE_{\le}:=\Big\{\forall i\le N:\ -\lambda x_N \bt \le \bX_i \le \varepsilon_0 x_N \bt\Big\}.
    \]
    Then
    \begin{align}
    \PP\Big(\sum_{i=1}^N \bX_i\ge x_N\bt,\ \calE_{\mathrm{box}}\Big)
    &\le
    \PP\Big(\sum_{i=1}^N \bX_i\ge x_N\bt,\ \calE_{\le}\Big)
    +\PP\Big(\sum_{i=1}^N \bX_i\ge x_N\bt,\ \calE_{\mathrm{box}}\cap \calE_{\le}^{\mathrm c}\Big).
    \label{eq:split-small-big}
    \end{align}
    By Lemma~\ref{lem:upperbound-small-part},
    \begin{equation}\label{eq:small-part-negl}
    \limsup_{\varepsilon_0\downarrow 0}\ \limsup_{N\to\infty}\frac{1}{x_N^\alpha}
    \log \PP\Big(\sum_{i=1}^N \bX_i\ge x_N\bt,\ \calE_{\le}\Big)=-\infty.
    \end{equation}
    For the second term in \eqref{eq:split-small-big}, on $\calE_{\mathrm{box}}\cap \calE_{\le}^{\mathrm c}$ define
    \[
    \Gamma:=\Big\{i\le N:\ \bX_i\not\le \varepsilon_0 x_N\bt\Big\} \text{ and note } |\Gamma|\ge 1.
    \]
    Fix $\varepsilon_0$ and choose $M=M(\varepsilon_0)\in\NN$ so large that
    \begin{equation}\label{eq:choose-M}
    (M+1)c(\varepsilon_0 \min_{1\le j\le k} t_j)^\alpha \ \ge\ \calI_{\calJ}(\bt)+1,
    \end{equation}
    where $c>0$ is as in Lemma~\ref{lem:finite-second}. For $i \in \Gamma$, we have that there is $j$ such that $X_{i,j}>\varepsilon_0 x_N t_j$, and so we obtain for all large $N$,
    \[
    \PP( |\Gamma|\ge M+1)
    \le \binom{N}{M+1}\Big(k\max_{1\le j\le k}\PP(X_{1,j}>\varepsilon_0 x_N t_j)\Big)^{M+1},
    \]
    hence, using $\log N=o(x_N^\alpha)$, Lemma~\ref{lem:finite-second} and \eqref{eq:choose-M},
    \begin{equation}\label{eq:m-large-negl}
    \limsup_{N\to\infty}\frac{1}{x_N^\alpha}\log \PP( |\Gamma|\ge M+1)\le -\calI_{\calJ}(\bt)-1.
    \end{equation}
    Consequently,
    \begin{align}
    \PP\Big(\sum_{i=1}^N \bX_i\ge x_N\bt,\ \calE_{\mathrm{box}}\cap \calE_{\le}^{\mathrm c}\Big)
    &\le \PP( |\Gamma|\ge M+1)+\sum_{m=1}^{M}\PP\Big(\sum_{i=1}^N \bX_i\ge x_N\bt,\ \calE_{\mathrm{box}},\ |\Gamma|=m\Big).
    \label{eq:split-m}
    \end{align}
    Fix $m\in\{1,\ldots,M\}$ and a set $A\subseteq\{1,\ldots,N\}$ with $|A|=m$.
    On $\{\Gamma=A\}\cap\calE_{\mathrm{box}}$ we have $\forall i\notin A:\ -\lambda x_N\bt\le \bX_i\le \varepsilon_0 x_N\bt$.
    Write
    \[
    \bS_{\mathrm{sm}}:=\sum_{i\notin A}\bX_i,\qquad \bS_{\mathrm{bg}}:=\sum_{i\in A}\bX_i.
    \]
    If $\bS_{\mathrm{sm}}+\bS_{\mathrm{bg}}\ge x_N\bt$ and $(\bS_{\mathrm{sm}})_j<\delta x_N t_j$ for all $j$, then
    \[
    \bS_{\mathrm{bg}}=(\bS_{\mathrm{sm}}+\bS_{\mathrm{bg}})-\bS_{\mathrm{sm}}
    \ge x_N\bt-\delta x_N\bt=(1-\delta)x_N\bt.
    \]
    Hence,
    \begin{equation}\label{eq:split-sm-bg}
    \{\bS_{\mathrm{sm}}+\bS_{\mathrm{bg}}\ge x_N\bt\}
    \subseteq
    \Big(\bigcup_{j=1}^k\{(\bS_{\mathrm{sm}})_j\ge \delta x_N t_j\}\Big)
    \ \cup\
    \{\bS_{\mathrm{bg}}\ge (1-\delta)x_N\bt\}.
    \end{equation}
    Therefore,
    \begin{align}
    &\PP\Big(\sum_{i=1}^N \bX_i\ge x_N\bt,\ \calE_{\mathrm{box}},\ \Gamma=A\Big)\nonumber\\
    &\le
    \sum_{j=1}^k
    \PP\Big(\sum_{i\notin A} X_{i,j}\ge \delta x_N t_j,\ \forall i\notin A:\ -\lambda x_N t_j\le X_{i,j}\le \varepsilon_0 x_N t_j\Big)\nonumber\\
    &\quad+
    \PP\Big(\sum_{i\in A}\bX_i\ge (1-\delta)x_N\bt,\ \forall i\le N:\ -\lambda x_N\bt\le \bX_i\le \lambda x_N\bt\Big).
    \label{eq:two-cases-A}
    \end{align}
    For each fixed $j$, apply Lemma~\ref{lem:upperbound-small-part} to the $j$-th probability above to obtain (using \cite[Lemma 1.2.15]{dembo2009techniques} and $|A^c| = N-m$),
    \begin{equation}\label{eq:small-complement-negl}
    \limsup_{\varepsilon_0\downarrow 0}\ \limsup_{N\to\infty}\frac{1}{x_N^\alpha}
    \log \sum_{j=1}^k
    \PP\Big(\sum_{i\notin A} X_{i,j}\ge \delta x_N t_j,\ \forall i\notin A:\ -\lambda x_N t_j\le X_{i,j}\le \varepsilon_0 x_N t_j\Big)=-\infty.
    \end{equation}
    We now treat the second term in \eqref{eq:two-cases-A}.
    Fix $\rho>0$ such that
    \begin{equation}\label{eq:rho-choice}
    M\rho\,\bfm 1 \ \le\ \delta \bt.
    \end{equation}
    Let
    \[
    \calD_\rho:=\prod_{j=1}^k \{-\lambda t_j+\ell\rho:\ \ell=0,1,\ldots,\lfloor 2\lambda t_j/\rho\rfloor\}\subseteq[-\lambda\bt,\lambda\bt].
    \]
    For $\by\in[-\lambda\bt,\lambda\bt]$ define $\pi_\rho(\by)\in\calD_\rho$ by
    \[
    (\pi_\rho(\by))_j:=-\lambda t_j+\rho\Big\lfloor\frac{y_j+\lambda t_j}{\rho}\Big\rfloor,\qquad j=1,\ldots,k,
    \]
    so that $\pi_\rho(\by)\le \by<\pi_\rho(\by)+\rho\bfm 1$.
    If $\sum_{i\in A}\bX_i\ge (1-\delta)x_N\bt$ and $\forall i\le N:\ -\lambda x_N\bt\le \bX_i\le \lambda x_N\bt$, then with $\by_i:=\bX_i/x_N\in[-\lambda\bt,\lambda\bt]$,
    \[
    \sum_{i\in A}\pi_\rho(\by_i)
    \ge (\sum_{i\in A}\by_i) - m\rho\,\bfm 1
    \ge (1-\delta)\bt - M\rho\,\bfm 1
    \ge (1-2\delta)\bt,
    \]
    where we used \eqref{eq:rho-choice}. Consequently,
    \begin{align}
    &\Big\{\sum_{i\in A}\bX_i\ge (1-\delta)x_N\bt,\ \forall i\le N:\ -\lambda x_N\bt\le \bX_i\le \lambda x_N\bt\Big\}\nonumber\\
    &\subseteq
    \bigcup_{\substack{(\bd^{(i)})_{i\in A}\in\calD_\rho^{\,m}\\ \sum_{i\in A} \bd^{(i)}\ge (1-2\delta)\bt}}
    \ \bigcap_{i\in A}\Big\{\bX_i\ge x_N \bd^{(i)},\ \forall \ell\le N:\ -\lambda x_N\bt\le \bX_\ell\le \lambda x_N\bt\Big\}.
    \label{eq:grid-union}
    \end{align}
    Since $\calD_\rho$ is finite, $\log|\calD_\rho|=O(1)$. Fix $m\in\{1,\ldots,M\}$ and a choice $(\bd^{(i)})_{i\in A}\in\calD_\rho^{\,m}$.
    Dropping the box constraint and using independence, we get
    \[
    \PP\Big(\bigcap_{i\in A}\big\{\bX_i\ge x_N \bd^{(i)},\ \forall \ell\le N:\ -\lambda x_N\bt\le \bX_\ell\le \lambda x_N\bt\big\}\Big)
    \le \prod_{i\in A}\PP\big(\bX_1\ge x_N \bd^{(i)}\big).
    \]
    By Definition~\ref{def:multiSE},
    \[
    \lim_{N\to\infty}\frac{1}{x_N^\alpha}\log \PP\big(\bX_1\ge x_N \bd^{(i)}\big)
    =-\calJ\big((\bd^{(i)})_+\big),
    \]
    whenever $(\bd^{(i)})_+\neq \bfm 0$, while if $(\bd^{(i)})_+=\bfm 0$ we trivially have
    $\limsup_{N\to\infty}x_N^{-\alpha}\log \PP(\bX_1\ge x_N \bd^{(i)})\le 0=-\calJ(\bfm 0)$.
    Thus, for every $\eta\in(0,1)$ and all $N$ large enough,
    \[
    \PP\big(\bX_1\ge x_N \bd^{(i)}\big)
    \le
    \exp\Big(-(1-\eta)\,\calJ\big((\bd^{(i)})_+\big)\,x_N^\alpha\Big),
    \]
    and, therefore,
    \[
    \PP\Big(\bigcap_{i\in A}\big\{\bX_i\ge x_N \bd^{(i)},\ \forall \ell\le N:\ -\lambda x_N\bt\le \bX_\ell\le \lambda x_N\bt\big\}\Big)
    \le
    \exp\Big(-(1-\eta)x_N^\alpha\sum_{i\in A}\calJ\big((\bd^{(i)})_+\big)\Big).
    \]
    Applying this bound in \eqref{eq:grid-union}, taking logarithms, dividing by $x_N^\alpha$, and using
    $\log|\calD_\rho|=O(1)$ and $\log\binom{N}{m}=o(x_N^\alpha)$ yields, after letting $\eta\downarrow 0$,
    \begin{align}
    &\limsup_{N\to\infty}\frac{1}{x_N^\alpha}\log 
    \PP\Big(\sum_{i\in A}\bX_i\ge (1-\delta)x_N\bt,\ \forall i\le N:\ -\lambda x_N\bt\le \bX_i\le \lambda x_N\bt\Big)\nonumber\\
    & \le
    -\inf_{\substack{(\bd^{(i)})_{i\in A}\in\calD_\rho^{\,m}\\ \sum_{i\in A} \bd^{(i)}\ge (1-2\delta)\bt}}
    \ \sum_{i\in A}\calJ\big((\bd^{(i)})_+\big).
    \label{eq:big-part-rate-fixedm}
    \end{align}
    If $\sum_{i\in A} \bd^{(i)}\ge (1-2\delta)\bt$, then $\sum_{i\in A} (\bd^{(i)})_+\ge (1-2\delta)\bt$. Therefore,
    \[
    \inf_{\substack{(\bd^{(i)})_{i\in A}\in\calD_\rho^{\,m}\\ \sum_{i\in A} \bd^{(i)}\ge (1-2\delta)\bt}}
    \ \sum_{i\in A}\calJ\big((\bd^{(i)})_+\big)
    \ \ge\
    \inf\Big\{\sum_{r=1}^m \calJ\big(\bu^{(r)}\big):\ \bu^{(r)}\in\bbR_+^k,\ \sum_{r=1}^m \bu^{(r)}\ge (1-2\delta)\bt\Big\}.
    \]
    Taking the infimum over $m\in\NN$ yields by Lemma \ref{lem:Ij_k_summands},
    \begin{align}\label{eq:big-part-rate}
    &\limsup_{N\to\infty}\frac{1}{x_N^\alpha}\log 
    \sum_{m=1}^{M}\ \sum_{\substack{A\subseteq\{1,\ldots,N\}\\ |A|=m}}
    \PP\Big(\sum_{i\in A}\bX_i\ge (1-\delta)x_N\bt,\ \forall i\le N:\ -\lambda x_N\bt\le \bX_i\le \lambda x_N\bt\Big)\\
    & \le -\calI_{\calJ}\big((1-2\delta)\bt\big).
    \end{align}
    Combining \eqref{eq:split-small-big}, \eqref{eq:small-part-negl}, \eqref{eq:split-m}, \eqref{eq:m-large-negl}, \eqref{eq:two-cases-A}, \eqref{eq:small-complement-negl} and \eqref{eq:big-part-rate}, we obtain
    \begin{equation}\label{eq:upper-delta}
    \limsup_{\varepsilon_0\downarrow 0}\ \limsup_{N\to\infty}\frac{1}{x_N^\alpha}
    \log \PP\Big(\sum_{i=1}^N \bX_i\ge x_N\bt,\ \calE_{\mathrm{box}}\Big)
    \le -\calI_{\calJ}\big((1-2\delta)\bt\big).
    \end{equation}
    Since $\calI_{\calJ}$ is $\alpha$-homogeneous (by $\alpha$-homogeneity of $\calJ$ and Lemma \ref{lem:homogeneity_barJ}), we have
    $\calI_{\calJ}((1-2\delta)\bt)=(1-2\delta)^\alpha \calI_{\calJ}(\bt)$ and thus
    $\calI_{\calJ}((1-2\delta)\bt)\uparrow \calI_{\calJ}(\bt)$ as $\delta\downarrow 0$.
    Letting $\delta\downarrow 0$ in \eqref{eq:upper-delta} and using \eqref{splitwithbox} and \eqref{eq:eboxc-rate} yields
    \[
    \limsup_{N\to\infty}\frac{1}{x_N^\alpha}\log \PP\Big(\frac{1}{x_N}\sum_{i=1}^N \bX_i\ge \bt\Big)
    \le -\calI_{\calJ}(\bt).
    \]
    This completes the proof.
\end{proof}

\begin{proof}[Proof of Theorem \ref{thm:LDP-multi}]
The result follows immediately from Lemma \ref{lem:lower-bound} and Lemma \ref{lem:upper-bound-multi}.
\end{proof}

\subsection{Proofs of the applications}

\begin{proof}[Proof of Lemma~\ref{lem:density-implies-SE}]
    Fix $\bt\in\bbR^k$ with at least one strictly positive component and set
    \[
    m(\bt):=\inf_{\bs\ge \bt_+}\overline{\calJ}(\bs)\in(0,\infty].
    \]
    
    \noindent\emph{Lower bound.}
    Fix $\delta\in(0,1)$. Choose $\bs^\delta\ge \bt_+$ such that $\overline{\calJ}(\bs^\delta)\le m(\bt)+\delta$ and, in addition, $\bs^\delta$ is a continuity point of $\overline{\calJ}$ (this is possible since $\overline{\calJ}$ is lower semicontinuous, hence continuous on a dense set).
    By continuity of $\overline{\calJ}$ at $\bs^\delta$, there exists $\varepsilon=\varepsilon_\delta>0$ such that for the box $B^\delta:= \{\bx\in \RR_+^k: \bs^\delta \le \bx \le \bs^\delta+\varepsilon\}$, we have $\sup_{\bu\in B^\delta}\overline{\calJ}(\bu)\le \overline{\calJ}(\bs^\delta)+\delta$. For $x$ large enough (so that $\| (x\bu)_+\|\ge R$ for all $\bu\in B^\delta$), using the lower density bound and $\alpha$-homogeneity,
    \[
    \PP(\bY\ge x\bt)
    \ \ge\
    \PP(\bY\in xB^\delta)
    \ \ge\
    \int_{xB^\delta} p_1(\bx)\,e^{-\overline{\calJ}(\bx)}\,\dd \bx
    =
    x^k\int_{B^\delta} p_1(x\bu)\,e^{-x^\alpha \overline{\calJ}(\bu)}\,\dd \bu.
    \]
    Hence,
    \begin{align}
    \PP(\bY\ge x\bt)
    &\ge
    x^k\,\vol(B^\delta)\,
    \inf_{\bu\in B^\delta}p_1(x\bu)\,
    \exp\Big(-x^\alpha\sup_{\bu\in B^\delta}\overline{\calJ}(\bu)\Big) \notag\\
    &\ge
    x^k\,\vol(B^\delta)\,
    \inf_{\bu\in B^\delta}p_1(x\bu)\,
    e^{-x^\alpha(\overline{\calJ}(\bs^\delta)+\delta)}. \notag
    \end{align}
    Dividing by $x^\alpha$ and letting $x\to\infty$, the assumptions on $p_1$ imply
    \[
    \liminf_{x\to\infty}\frac{1}{x^\alpha}\log\PP(\bY\ge x\bt)
    \ \ge\
    -(\overline{\calJ}(\bs^\delta)+\delta)
    \ \ge\
    -(m(\bt)+2\delta).
    \]
    Letting $\delta\downarrow0$ yields
    \[
    \liminf_{x\to\infty}\frac{1}{x^\alpha}\log\PP(\bY\ge x\bt)\ \ge\ -m(\bt).
    \]
    
    \noindent\emph{Upper bound.}
    Fix $M>1$ and define the box $K_{M,x}:=\{\bx\in\bbR^k:\forall j,\ |x_j|\le Mx\|\bt\|\}$ and
    $D_{M,x}:=\{\bx\in K_{M,x}:\bx\ge x\bt\}$. Then
    \begin{equation}\label{split-with-Kbox}
    \PP(\bY\ge x\bt)\ \le\ \PP(\bY\in D_{M,x})+\PP(\bY\notin K_{M,x}).
    \end{equation}
    Considering the second term on the r.h.s. of \eqref{split-with-Kbox}, by similar arguments as in the proof of Lemma \ref{lem:finite-second} for the lower bound and the upper density bound for the right tails, there exist $c,C>0$ such that $\PP(\bY\notin K_{M,x})\le C e^{-c M^\alpha x^\alpha}$ for all $x$.
    Choose $M$ so large that $cM^\alpha>m(\bt)$, then
    \[
    \limsup_{x\to\infty}\frac{1}{x^\alpha}\log \PP(\bY\notin K_{M,x})\ \le\ -cM^\alpha\ <\ -m(\bt).
    \]
    For the first term on the r.h.s. of \eqref{split-with-Kbox}, for $x$ large enough the upper density bound applies on $D_{M,x}$, hence
    \[
    \PP(\bY\in D_{M,x})
    =\int_{D_{M,x}} f(\bx)\,\dd \bx
    \ \le\
    \int_{D_{M,x}} p_2(\bx_+)\,e^{-\overline{\calJ}(\bx_+)}\,\dd \bx.
    \]
    On $D_{M,x}$ we have $\bx_+\ge x\bt_+$, hence $\overline{\calJ}(\bx_+)\ge \inf_{\bs\ge x\bt_+}\overline{\calJ}(\bs)=x^\alpha m(\bt)$ by $\alpha$-homogeneity. Therefore,
    \[
    \PP(\bY\in D_{M,x})
    \ \le\
    \vol(D_{M,x})\,
    \sup_{\bx\in D_{M,x}} p_2(\bx_+)\,
    e^{-x^\alpha m(\bt)}.
    \]
    Since $\vol(D_{M,x})\le (2Mx\|\bt\|)^k$ and $(\bx_+/x)_{\bx\in D_{M,x}}$ ranges over a fixed compact subset of $\bbR_+^k$,
    the assumptions on $p_2$ give
    \[
    \limsup_{x\to\infty}\frac{1}{x^\alpha}\log \PP(\bY\in D_{M,x})\ \le\ -m(\bt).
    \]
    Combining the two terms on the r.h.s. of \eqref{split-with-Kbox} yields
    \[
    \limsup_{x\to\infty}\frac{1}{x^\alpha}\log \PP(\bY\ge x\bt)\ \le\ -m(\bt).
    \]
    Together with the lower bound, this proves the claim.
\end{proof}

\begin{proof}[Proof of Lemma \ref{lem:power-gaussian-density}]
    Let \(\bU=(|Y_1|,\ldots,|Y_k|)\). By \cite[Section 3]{chakraborty2013multivariate}, for \(\bu\in\bbR_+^k\),
    \[
    f_{\bU}(\bu)=\sum_{\beps\in\{-1,1\}^k}\frac{1}{(2\pi)^{k/2}(\det\bSigma)^{1/2}}
    \exp\Big(-\frac12\,(\beps\odot \bu)^\top\bSigma^{-1}(\beps\odot \bu)\Big),
    \]
    where $f_{\bU}$ is the density of $\bU$. Set \(\bZ=\Phi(\bU)\) with \(\Phi(\bu)=\bu^{q}\) for $\bu\in \bbR_+^k$. Then \(\Phi^{-1}(\bz)=\bz^{1/q}\) and
    \[
    \det D\Phi^{-1}(\bz)=\prod_{i=1}^k \frac{1}{q}\,z_i^{1/q-1}
    =\frac{1}{q^k\prod_{i=1}^k z_i^{(q-1)/q}}.
    \]
    Hence, \(f_{\bZ}(\bz)=f_{\bU}(\bz^{1/q})\det D\Phi^{-1}(\bz)\), which gives the formula for the density of \(\bZ\). \eqref{powersfit} follows since the prefactor is polynomial and the sum has finitely many terms.
\end{proof}

\begin{proof}[Proof of Theorem \ref{thm:abs-powers-gauss}]
    Let \(\bG_1,\bG_2,\ldots\) be i.i.d.\ \(\mathcal{N}(\bfm0,\bSigma)\) in \(\bbR^k\) and set
    \begin{align*}
    \bW_i &= (|G_{i,1}|^q,\ldots,|G_{i,k}|^q),\quad
    \bmu_q := \big(\bSigma_{11}^{q/2}M_q,\ldots,\bSigma_{kk}^{q/2}M_q\big),\quad
    \bX_i := \bW_i-\bmu_q, \\
    \bY_N &= \frac1N\sum_{i=1}^N \bX_i.
    \end{align*}
    Assume first that \(\bSigma\) is positive definite. By Lemma \ref{lem:power-gaussian-density} and \eqref{eq:density-bounds-J}, \(\bX_1=\bW_1-\bmu_q\) has a Lebesgue density \(f_{\bX_1}\) on \(\bbR^k\) with
    \[
    p_1(\bx_+)\,e^{-\overline{\calJ}(\bx_+)}\le f_{\bX_1}(\bx)\le p_2(\bx_+)\,e^{-\overline{\calJ}(\bx_+)}\qquad (\|\bx_+\|\ \text{large})
    \]
    for some positive polynomials \(p_1,p_2\) and
    \[
    \overline{\calJ}(\bz)=\min_{\beps\in\{\pm1\}^k}\frac12\,(\beps\odot \bz^{1/q})^\top\bSigma^{-1}(\beps\odot \bz^{1/q}),\qquad \bz\in\bbR_+^k.
    \]
    Each marginal \(W_{1,j}\) has $1$-dimensional stretched-exponential tail with index \(\alpha=2/q\in(0,1)\).
    Lemma \ref{lem:density-implies-SE} applied to \(\bX_1\) with this \(\overline{\calJ}\) and \(\alpha\) yields that \(\bX_1\) has stretched-exponential tail of rate \((\calJ,\alpha)\) in the sense of Definition \ref{def:multiSE}, where \(\calJ(\bt)=\inf_{\bs\ge \bt}\overline{\calJ}(\bs)\). We can therefore apply Theorem \ref{thm:LDP-multi} with \(x_N=N\) to the sequence \((\bX_i)_{i\ge1}\). This gives, for every \(\bt\in(0,\infty)^k\),
    \[
    \lim_{N\to\infty}\frac{1}{N^{2/q}}
    \log\PP\Big(\frac1N\sum_{i=1}^N\bX_i\ge \bt\Big)
    = -\calI_{\calJ}(\bt).
    \]
    This proves the theorem in the positive definite case.\\
    
    Now assume that \(\bSigma\) is only positive semidefinite. For \(\delta>0\) set \(\bSigma_\delta:=\bSigma+\delta\,\Id_k\), which is positive definite. Let \(\bG_i^{(\delta)}\sim\mathcal{N}(\bfm0,\bSigma_\delta)\) be i.i.d.\ and define
    \begin{align*}
    \bW_i^{(\delta)} &= (|G_{i,1}^{(\delta)}|^q, \ldots, |G_{i,k}^{(\delta)}|^q),\quad
    \bmu_{q,\delta} := \big((\bSigma_{11}+\delta)^{q/2}M_q,\ldots,(\bSigma_{kk}+\delta)^{q/2}M_q\big), \\
    \bX_i^{(\delta)} &:= \bW_i^{(\delta)}-\bmu_{q,\delta},\qquad
    \bY_N^{(\delta)} := \frac1N\sum_{i=1}^N \bX_i^{(\delta)}.
    \end{align*}
    By the first part applied with \(\bSigma_\delta\), for every \(\bt\in(0,\infty)^k\),
    \begin{equation}\label{eq:LDP-Sigma-delta}
    \lim_{N\to\infty}\frac{1}{N^{2/q}}
    \log\PP\big(\bY_N^{(\delta)}\ge \bt\big)
    = -\calI_{\calJ_\delta}(\bt),
    \end{equation}
    where
    \[
    \overline{\calJ_\delta}(\bz)
    :=\min_{\beps\in\{\pm1\}^k}\frac12\,(\beps\odot \bz^{1/q})^\top\bSigma_\delta^{-1}(\beps\odot \bz^{1/q}),
    \qquad \bz\in\bbR_+^k,
    \]
    \[
    \calJ_\delta(\bt):=\inf_{\bs\ge \bt}\overline{\calJ_\delta}(\bs),\qquad \bt\in\bbR_+^k,
    \]
    and
    \[
    \calI_{\calJ_\delta}(\bt)
    :=\min\Big\{\sum_{r=1}^k \calJ_\delta\big(\bt^{(r)}\big):
    \bt^{(r)}\in\bbR_+^k,\ \sum_{r=1}^k \bt^{(r)}=\bt\Big\}.
    \]
    For each \(\bx\in\bbR^k\) we have
    \[
    \bx^\top\bSigma_\delta^{-1}\bx\ \uparrow\ \bx^\top\bSigma^+\bx
    \qquad(\delta\downarrow0),
    \]
    where we adopt the convention that \(\bx^\top\bSigma^+\bx=+\infty\) if \(\bx\notin\operatorname{range}(\bSigma)\).
    Hence, for every \(\bz\in\bbR_+^k\),
    \[
    \overline{\calJ_\delta}(\bz)\uparrow
    \overline{\calJ}(\bz)
    :=\min_{\beps\in\{\pm1\}^k}\frac12\,(\beps\odot \bz^{1/q})^\top\bSigma^+(\beps\odot \bz^{1/q}).
    \]
    Consequently, for every \(\bt\in\bbR_+^k\),
    \[
    \calJ_\delta(\bt)\uparrow \calJ(\bt) =\inf_{\bs\ge \bt}\overline{\calJ}(\bs),
    \]
    and, for every \(\bt\in\bbR_+^k\), by monotone convergence,
    \[
    \calI_{\calJ_\delta}(\bt)\uparrow \calI_{\calJ}(\bt)
    =\min\Big\{\sum_{r=1}^k \calJ\big(\bt^{(r)}\big):
    \bt^{(r)}\in\bbR_+^k,\ \sum_{r=1}^k \bt^{(r)}=\bt\Big\}.
    \]
To transfer the asymptotics from \(\bY_N^{(\delta)}\) to \(\bY_N\), we couple the sequences on one probability space. Let \(\bH_i\sim\mathcal{N}(\bfm0,\bSigma)\) and \(\bZ_i\sim\mathcal{N}(\bfm0,\Id_k)\) be independent i.i.d.\ for each \(i\). Set
\[
\bG_i:=\bH_i\qquad\text{and}\qquad\bG_i^{(\delta)}:=\bH_i+\sqrt{\delta}\,\bZ_i,
\]
so that \(\bG_i\sim\mathcal{N}(\bfm0,\bSigma)\) and \(\bG_i^{(\delta)}\sim\mathcal{N}(\bfm0,\bSigma_\delta)\). Define \(\bX_i,\bX_i^{(\delta)},\bY_N,\bY_N^{(\delta)}\) in terms of \(\bG_i,\bG_i^{(\delta)}\) as above. We show that for every \(\eps>0\),
\begin{equation}\label{eq:exp-eq}
\lim_{\delta\downarrow0}\ \limsup_{N\to\infty}
\frac{1}{N^{2/q}}\log \PP\big(\|\bY_N^{(\delta)}-\bY_N\|_\infty>\eps\big)=-\infty.
\end{equation}
Fix \(\delta>0\) and \(j\in\{1,\ldots,k\}\). By the mean value theorem applied to \(x \mapsto |x|^q\), there exists a random variable \(\theta\) with values between \(H_{1,j}\) and \(H_{1,j}+\sqrt{\delta}Z_{1,j}\) such that
\[
\big| |H_{1,j}+\sqrt{\delta}Z_{1,j}|^q - |H_{1,j}|^q \big|
= q|\theta|^{q-1}\sqrt{\delta}|Z_{1,j}|.
\]
Hence,
\[
|X_{1,j}^{(\delta)}-X_{1,j}|
\le q|\theta|^{q-1}\sqrt{\delta}|Z_{1,j}|
+\big|(\bSigma_{jj}+\delta)^{q/2}-\bSigma_{jj}^{q/2}\big|\,M_q.
\]
Using \(|\theta| \le |H_{1,j}| + \sqrt{\delta}|Z_{1,j}|\) and convexity of \(x \mapsto x^{q-1}\) (since $q>2$), we obtain
\[
|X_{1,j}^{(\delta)}-X_{1,j}|
\le C\sqrt{\delta}\Big(|H_{1,j}|^{q-1}|Z_{1,j}|+\delta^{(q-1)/2}|Z_{1,j}|^q\Big)
+\big|(\bSigma_{jj}+\delta)^{q/2}-\bSigma_{jj}^{q/2}\big|\,M_q.
\]
Since \(\big|(\bSigma_{jj}+\delta)^{q/2}-\bSigma_{jj}^{q/2}\big|\,M_q\to 0\) as \(\delta\downarrow0\), this deterministic shift can be absorbed into the estimates below by adjusting constants. The probability \(\PP(|X_{1,j}^{(\delta)}-X_{1,j}| > t)\) is bounded by the sum of probabilities that each term on the right-hand side exceeds \(t/2\).
For the first term, \(C\sqrt{\delta}\,|H_{1,j}|^{q-1}|Z_{1,j}| > t/2\) implies \(\max\{|H_{1,j}|, |Z_{1,j}|\} > c_1 \delta^{-1/(2q)}t^{1/q}\). Using the standard Gaussian tail estimate \(\PP(|Z|>u) \le e^{-u^2/2}\) for $u$ large enough, this probability is bounded by \(C_2 \exp(-c_3 \delta^{-1/q} t^{2/q})\).
The second term is controlled similarly by \(\PP(|Z_{1,j}| > c_4 \delta^{-1/2} t^{1/q})\).
Combining these,
\[
\PP\big(|X_{1,j}^{(\delta)}-X_{1,j}|>t\big)
\le C'\exp\big(-\gamma(\delta)\,t^{2/q}\big),
\]
where \(\gamma(\delta) \to \infty\) as \(\delta\downarrow 0\).
Applying the $1$-dimensional stretched-exponential bounds (e.g., \cite{nagaev1969integral}) to the centered sums of these differences yields
\[
\limsup_{N\to\infty}\frac{1}{N^{2/q}} \log \PP\Big(\Big|\frac1N\sum_{i=1}^N (X_{i,j}^{(\delta)}-X_{i,j})\Big|>\eps\Big)
\le - c(\eps)\gamma(\delta).
\]
Taking a union bound over \(j\) and letting \(\delta\downarrow0\) confirms \eqref{eq:exp-eq}. Fix \(\bt\in(0,\infty)^k\) and \(\varepsilon\in(0,\min_j t_j)\). For every \(\delta>0\) and \(N\),
\[
\PP\big(\bY_N\ge\bt\big)
\le \PP\big(\bY_N^{(\delta)}\ge \bt-\varepsilon\bfm1\big)
+ \PP\big(\|\bY_N^{(\delta)}-\bY_N\|_\infty>\varepsilon\big).
\]
Apply \cite[Lemma 1.2.15]{dembo2009techniques} to get
\[
\limsup_{N\to\infty}\frac{1}{N^{2/q}}\log\PP\big(\bY_N\ge\bt\big)
\le \max\Big\{
-\calI_{\calJ_\delta}(\bt-\varepsilon\bfm1),\ 
\limsup_{N\to\infty}\frac{1}{N^{2/q}}\log\PP\big(\|\bY_N^{(\delta)}-\bY_N\|_\infty>\varepsilon\big)
\Big\}.
\]
By the exponential bound for \(\|\bY_N^{(\delta)}-\bY_N\|_\infty\) and \eqref{eq:exp-eq}, for every fixed \(\varepsilon>0\), we can choose \(\delta\) small enough such that the decay rate of the difference term exceeds \(\calI_{\calJ_\delta}(\bt-\varepsilon\bfm1)\). Consequently,
\[
\limsup_{N\to\infty}\frac{1}{N^{2/q}}\log\PP\big(\bY_N\ge\bt\big)
\le -\calI_{\calJ_\delta}(\bt-\varepsilon\bfm1).
\]
Letting \(\delta\downarrow0\) and then \(\varepsilon\downarrow0\) (using lower semicontinuity of \(\calI_{\calJ}\)) yields
\[
\limsup_{N\to\infty}\frac{1}{N^{2/q}}\log\PP\big(\bY_N\ge\bt\big)
\le -\calI_{\calJ}(\bt).
\]
For the lower bound, fix again \(\bt\in(0,\infty)^k\) and \(\eta\in(0,1)\), and set \(\varepsilon:=\eta\,\min_j t_j\) so that \((1+\eta)\bt-\varepsilon\bfm1\ge \bt\). For every \(\delta>0\) and \(N\),
\[
\PP\big(\bY_N\ge\bt\big) \ge \PP\big(\bY_N^{(\delta)}\ge (1+\eta)\bt,\ \|\bY_N^{(\delta)}-\bY_N\|_\infty\le\varepsilon\big) \ge \PP\big(\bY_N^{(\delta)}\ge (1+\eta)\bt\big)
- \PP\big(\|\bY_N^{(\delta)}-\bY_N\|_\infty>\varepsilon\big).
\]
By \eqref{eq:LDP-Sigma-delta},
\[
\lim_{N\to\infty}\frac{1}{N^{2/q}}\log\PP\big(\bY_N^{(\delta)}\ge (1+\eta)\bt\big)
= -\calI_{\calJ_\delta}\big((1+\eta)\bt\big).
\]
By \eqref{eq:exp-eq}, for sufficiently small \(\delta\), the first term dominates:
\[
\liminf_{N\to\infty}\frac{1}{N^{2/q}}\log\PP\big(\bY_N\ge\bt\big)
\ge -\calI_{\calJ_\delta}\big((1+\eta)\bt\big).
\]
Letting \(\delta\downarrow0\) yields
\[
\liminf_{N\to\infty}\frac{1}{N^{2/q}}\log\PP\big(\bY_N\ge\bt\big)
\ge -\calI_{\calJ}\big((1+\eta)\bt\big),
\]
and finally letting \(\eta\downarrow0\) and using the \(\alpha\)-homogeneity of \(\calI_{\calJ}\) gives
\[
\liminf_{N\to\infty}\frac{1}{N^{2/q}}\log\PP\big(\bY_N\ge\bt\big)
\ge -\calI_{\calJ}(\bt).
\]
Combining the upper and lower bounds completes the proof.
\end{proof}

\begin{proof}[Proof of Corollary \ref{cor:rho-threshold}]
    This follows directly, as a special case, from Theorem \ref{thm:abs-powers-gauss}.
\end{proof}

\begin{lemma}\label{lem:I_not_convex_nor_concave_q_gt_2_rho_pos}
    Let $q>2$ and $\rho\in(0,1)$. Define
    \[
    \overline{\calJ}(z_1,z_2)
    :=\frac{z_1^{2/q}+z_2^{2/q}-2\rho\,z_1^{1/q}z_2^{1/q}}{2(1-\rho^2)},
    \qquad (z_1,z_2)\in\bbR_+^2,
    \]
    \[
    \calJ(\bt):=\inf_{\bs\ge \bt}\overline{\calJ}(\bs),\qquad \bt\in\bbR_+^2,
    \]
    and
    \[
    \calI_{\calJ}(\bt)
    :=\min\Big\{\calJ(\bu)+\calJ(\bv):\ \bu,\bv\in\bbR_+^2,\ \bu+\bv=\bt\Big\},
    \qquad \bt\in\bbR_+^2.
    \]
    Then $\calI_{\calJ}$ is neither convex nor concave on $\bbR_+^2$.
\end{lemma}
    
\begin{proof}
    Fix $t\ge 0$ and let $s_1\ge t$, $s_2\ge 0$. Writing $a:=s_1^{1/q}$ and $b:=s_2^{1/q}$,
    \[
    a^2+b^2-2\rho ab=(b-\rho a)^2+(1-\rho^2)a^2\ge (1-\rho^2)a^2,
    \]
    and so $\overline{\calJ}(s_1,s_2)\ge a^2/2=s_1^{2/q}/2$. Therefore,
    \[
    \calJ(t,0)=\inf_{\substack{s_1\ge t\\ s_2\ge 0}}\overline{\calJ}(s_1,s_2)\ge \frac{t^{2/q}}{2}.
    \]
    Choosing $(s_1,s_2)=(t,\rho^q t)$ gives $b=\rho a$ and thus $\overline{\calJ}(t,\rho^q t)=t^{2/q}/2$, so
    \[
    \calJ(t,0)=\frac{t^{2/q}}{2},\qquad \calJ(0,t)=\frac{t^{2/q}}{2}.
    \]
    Set $p:=2/q\in(0,1)$. For $t\ge 0$,
    \[
    \calI_{\calJ}(t,0)=\min_{u\in[0,t]}\Big(\calJ(u,0)+\calJ(t-u,0)\Big)
    =\frac12\min_{u\in[0,t]}\Big(u^p+(t-u)^p\Big)=\frac{t^p}{2},
    \]
    since $(x+y)^p\le x^p+y^p$ for all $x,y\ge 0$, with equality at $u\in\{0,t\}$.
    
    \emph{Not convex.} Using $\calI_{\calJ}(0,0)=0$ and $\calI_{\calJ}(1,0)=1/2$,
    \[
    \calI_{\calJ}\Big(\frac12,0\Big)=\frac12\Big(\frac12\Big)^p=2^{-(1+p)}>\frac14
    =\frac12\,\calI_{\calJ}(1,0)+\frac12\,\calI_{\calJ}(0,0),
    \]
    because $p<1$. Hence, $\calI_{\calJ}$ is not convex.
    
    \emph{Not concave.} Let $\varepsilon:=\rho^q/(1+\rho^q)\in(0,1)$, so that $\varepsilon^{1/q}=\rho(1-\varepsilon)^{1/q}$. Then
    \[
    \calI_{\calJ}(\varepsilon,1-\varepsilon)\le \calJ(\varepsilon,1-\varepsilon)\le \overline{\calJ}(\varepsilon,1-\varepsilon)
    =\frac{(1-\varepsilon)^p}{2}<\frac12.
    \]
    But $\varepsilon(1,0)+(1-\varepsilon)(0,1)=(\varepsilon,1-\varepsilon)$ and
    $\calI_{\calJ}(1,0)=\calI_{\calJ}(0,1)=1/2$, hence concavity would imply
    \[
    \calI_{\calJ}(\varepsilon,1-\varepsilon)\ge \varepsilon \calI_{\calJ}(1,0)+(1-\varepsilon)\calI_{\calJ}(0,1)=\frac12,
    \]
    a contradiction. Thus $\calI_{\calJ}$ is not concave.
\end{proof}

One can show that in the multivariate Weibull example, a deviation is typically realized by exactly one vector. The next lemma gives an example where a deviation is realized by two vectors.

\begin{lemma}[A deviation induced by two big jumps]\label{lem:two-big-jumps}
    Let $R$ be Weibull with parameter $1/2$, i.e.\ $\PP(R\ge t)=\exp(-t^{1/2})$ for $t\ge0$.
    Fix $\eps\in(0,1)$. Let $\bV\in\{(1,\eps),(\eps,1)\}$ with
    $\PP(\bV=(1,\eps))=\PP(\bV=(\eps,1))=\frac12$.
    Let $S_1,S_2$ be i.i.d.\ Rademacher variables, independent of $(R,\bV)$, and set
    \[
    \bX=(X_1,X_2):=(R V_1 S_1,\ R V_2 S_2).
    \]
    Then $\bX$ satisfies \Cref{def:multiSE} with $\alpha=\frac12$ and Assumption \ref{assu-LT} with
    \[
    \calJ(t_1,t_2)
    =\min\Big\{(\max\{t_1,t_2/\eps\})^{1/2},\ (\max\{t_1/\eps,t_2\})^{1/2}\Big\},
    \qquad (t_1,t_2)\in\bbR_+^2.
    \]
    In particular, $\calJ(1,1)=\eps^{-1/2}$.
    Let $(\bX_i)_{i\ge1}$ be i.i.d.\ copies of $\bX$ and take $x_N=N$. Then
    \[
    \liminf_{N\to\infty}\frac{1}{N^{1/2}}\log \PP\left(\frac{1}{N}\sum_{i=1}^N \bX_i\ge (1,1)\right)\ge -2.
    \]
    Consequently, whenever $\eps<1/4$, the lower bound $-2$ is strictly larger than $-\calJ(1,1)=-\eps^{-1/2}$, so $\calJ$ is not the correct rate for this deviation.
    \end{lemma}
    
    \begin{proof}
    For $j=1,2$ and $t$ large,
    \[
    \PP(X_j\ge t)
    =\frac14\,\PP(R\ge t)+\frac14\,\PP(R\ge t/\eps)
    \le \exp(-c\,t^{1/2})
    \]
    for some $c=c(\eps)>0$. Moreover,
    \begin{align*}
    \PP(X_1\ge t_1,\ X_2\ge t_2)
    &=\frac18\,\PP\left(R\ge \max\{t_1,t_2/\eps\}\right)
     +\frac18\,\PP\left(R\ge \max\{t_1/\eps,t_2\}\right)\\
    &=\frac18\,\exp\Big(-(\max\{t_1,t_2/\eps\})^{1/2}\Big)
      +\frac18\,\exp\Big(-(\max\{t_1/\eps,t_2\})^{1/2}\Big).
    \end{align*}
    Hence, $\bX$ satisfies \Cref{def:multiSE} with $\alpha=\frac12$ and
    \[
    \calJ(t_1,t_2)
    =\min\Big\{(\max\{t_1,t_2/\eps\})^{1/2},\ (\max\{t_1/\eps,t_2\})^{1/2}\Big\},
    \]
    so $\calJ(1,1)=\eps^{-1/2}$. Assumption \ref{assu-LT} holds by symmetry (thus \Cref{LT} holds). Now let $\big((R_i,\bV_i,S_{i,1},S_{i,2})\big)_{i\ge1}$ be i.i.d.\ copies of $(R,\bV,S_1,S_2)$ and set
\[
\bX_i=(X_{i,1},X_{i,2})
:=\big(R_i(\bV_i)_1S_{i,1},\ R_i(\bV_i)_2S_{i,2}\big),\qquad i\ge1.
\]
Fix $\delta\in(0,1)$. Then
\begin{align*}
&\PP\left(\frac{1}{N}\sum_{i=1}^N X_{i,1}\ge1,\ \frac{1}{N}\sum_{i=1}^N X_{i,2}\ge1\right)\\
&\ge
\PP\left(\frac{1}{N}\sum_{i=3}^N X_{i,1}\ge-\delta,\ \frac{1}{N}\sum_{i=3}^N X_{i,2}\ge-\delta\right)
\times
\PP\left(S_{1,1}=1,\ S_{1,2}=1,\ \bV_1=(1,\eps),\ R_1\ge N(1+\delta)\right)\\
&\qquad\times
\PP\left(S_{2,1}=1,\ S_{2,2}=1,\ \bV_2=(\eps,1),\ R_2\ge N(1+\delta)\right).
\end{align*}
    The first factor tends to $1$ by the law of large numbers. Each of the last two factors equals
    \[
    \frac18\,\PP\big(R\ge N(1+\delta)\big)
    =\frac18\,\exp\big(-(1+\delta)^{1/2}N^{1/2}\big).
    \]
    Therefore,
    \[
    \liminf_{N\to\infty}\frac{1}{N^{1/2}}\log \PP\left(\frac{1}{N}\sum_{i=1}^N \bX_i\ge (1,1)\right)
    \ge -2(1+\delta)^{1/2}.
    \]
    Letting $\delta\downarrow0$ yields the asserted bound $-2$. If $\eps<1/4$, then $\eps^{-1/2}>2$, hence $-2>-\eps^{-1/2}=-\calJ(1,1)$, showing that $\calJ$ alone cannot be the correct rate for this event; the correct rate is given by $\calI_{\calJ}$ as in \Cref{thm:LDP-multi}.
\end{proof}    

\begin{proof}[Proof of Theorem \ref{thm: LDP for finite stiefel evals}]
    Let $\mathbf{g}_1,\mathbf{g}_2,\ldots$ be i.i.d.\ standard Gaussian vectors in $\RR^m$ and let $\mathbf{G}_N$ be the $m\times N$ matrix with columns $\mathbf{g}_1,\ldots,\mathbf{g}_N$. It is well known (see for instance \cite[Lemma 3.1]{KabluchkoProchnoThale2019}) that
    \[
    \mathbf{V}_N\ \overset{d}{=}\ \Big(\frac{1}{N}\mathbf{G}_N\mathbf{G}_N^\top\Big)^{-\frac{1}{2}}\,\mathbf{G}_N,
    \]
    hence $\sqrt{N}\,\bv_i\overset{d}{=}\mathbf{A}_N \mathbf{g}_i$ with
    \[
    \mathbf{A}_N:=\Big(\frac{1}{N}\mathbf{G}_N\mathbf{G}_N^\top\Big)^{-\frac12}.
    \]
    Hence, for $\eps\in(0,1)$,
    \begin{align*}
    &\PP\Big(\frac{1}{N}\sum_{i=1}^N |\langle  \sqrt{N}\,\bv_i, \bu_1\rangle |^q - M_q  \ge t_1, \ldots,\ \frac{1}{N}\sum_{i=1}^N |\langle  \sqrt{N}\,\bv_i, \bu_k \rangle|^q - M_q \ge t_k\Big)\\
    &\le \sup_{\mathbf{D}\in B_\eps(\Id_m)}\ \PP\Big(\frac{1}{N}\sum_{i=1}^N |\langle \mathbf{D}\mathbf{g}_i, \bu_1\rangle |^q - M_q  \ge t_1, \ldots,\ \frac{1}{N}\sum_{i=1}^N |\langle \mathbf{D}\mathbf{g}_i, \bu_k \rangle|^q - M_q \ge t_k\Big)\\
    &\qquad\qquad +\ \PP\big(\mathbf{A}_N\notin B_\eps(\Id_m)\big).
    \end{align*}
    By Cram\'er's theorem (Proposition \ref{prop:cramer}) for $\frac1N \mathbf{G}_N\mathbf{G}_N^\top$ (at speed $N$) and contraction under $\bA\mapsto \bA^{-1/2}$ on positive definite matrices,
    \begin{equation}\label{eq:stiefel-gram-fast}
    \lim_{N\to\infty}\frac{1}{N^{2/q}}\log \PP\big(\mathbf{A}_N\notin B_\eps(\Id_m)\big)=-\infty.
    \end{equation}
    
    \noindent \emph{Upper bound.}
    Fix $\mathbf{D}\in B_\eps(\Id_m)$. For each $j \in \{1,\ldots,k\}$ set
    \[
    \sigma_j(\mathbf{D}):=\|\mathbf{D}^\top \bu_j\|,\qquad \widehat\bu_j(\mathbf{D}):=\frac{\mathbf{D}^\top \bu_j}{\|\mathbf{D}^\top \bu_j\|}\in\bbS^{m-1}.
    \]
    Then $\langle \mathbf{D}\mathbf{g}_i,\bu_j\rangle=\sigma_j(\mathbf{D})\langle \mathbf{g}_i,\widehat\bu_j(\mathbf{D})\rangle$ and hence
    \begin{align*}
    \frac{1}{N}\sum_{i=1}^N \big(|\langle \mathbf{D}\mathbf{g}_i,\bu_j\rangle|^q-M_q\big)\ge t_j
    \ \Longleftrightarrow\
    \frac{1}{N}\sum_{i=1}^N |\langle \mathbf{g}_i,\widehat\bu_j(\mathbf{D})\rangle|^q - M_q\ \ge\ \frac{t_j+M_q}{\sigma_j(\mathbf{D})^q}-M_q.
    \end{align*}
    Let $\bSigma(\mathbf{D}):=\big(\langle \widehat\bu_r(\mathbf{D}),\widehat\bu_s(\mathbf{D})\rangle\big)_{1\le r,s\le k}$. The matrix $\bSigma(\mathbf{D})$ is positive semidefinite and $\bSigma(\mathbf{D})\to \bSigma$ as $\eps\downarrow0$. By Theorem~\ref{thm:abs-powers-gauss} with the corresponding rate function $\calI_{\calJ}^{\,\bSigma(\mathbf{D})}$,
    \[
    \limsup_{N\to\infty}\frac{1}{N^{2/q}}\log
    \PP\Big(\frac{1}{N}\sum_{i=1}^N |\langle \mathbf{D}\mathbf{g}_i, \bu_\cdot\rangle |^q - M_q  \ge \bt \Big)
    \ \le\ -\calI_{\calJ}^{\,\bSigma(\mathbf{D})}(\bt^{(\eps)}(\mathbf{D})),
    \]
    where $\bt^{(\eps)}(\mathbf{D})$ has coordinates $t_j^{(\eps)}(\mathbf{D}):=\frac{t_j+M_q}{\sigma_j(\mathbf{D})^q}-M_q$. Continuity of $(\bSigma,\bt)\mapsto\calI_{\calJ}^{\,\bSigma}(\bt)$ and $\sigma_j(\mathbf{D})\to1$ give
    \[
    \sup_{\mathbf{D}\in B_\eps(\Id_m)}\ \calI_{\calJ}^{\,\bSigma(\mathbf{D})}(\bt^{(\eps)}(\mathbf{D}))
    \ \xrightarrow[\eps\downarrow0]{}\ \calI_{\calJ}^{\,\bSigma}(\bt).
    \]
    With \eqref{eq:stiefel-gram-fast} this yields
    \[
    \limsup_{N\to\infty}\frac{1}{N^{2/q}}\log
    \PP\Big(\frac{1}{N}\sum_{i=1}^N |\langle  \sqrt{N}\,\bv_i, \bu_\cdot\rangle |^q - M_q  \ge \bt\Big)
    \le -\calI_{\calJ}^{\,\bSigma}(\bt).
    \]
    
    \noindent \emph{Lower bound.}
    Define
    \[
    \mathbf{Y}_N:=\Big(\frac{1}{N}\sum_{i=1}^N |\langle  \sqrt{N}\,\bv_i, \bu_\ell\rangle |^q - M_q\Big)_{\ell=1}^k,\qquad
    \mathbf{Z}_N:=\Big(\frac{1}{N}\sum_{i=1}^N |\langle  \mathbf{g}_i, \bu_\ell\rangle |^q - M_q\Big)_{\ell=1}^k.
    \]
    We prove exponential equivalence at speed \(N^{2/q}\). Fix \(\delta\in(0,1)\). Whenever \(\mathbf{A}_N\in B_\delta(\Id_m)\), we have
    \[
    |\sigma_\ell(\mathbf{A}_N)^q-1|\le C\delta,\quad \|\widehat\bu_\ell(\mathbf{A}_N)-\bu_\ell\|_2\le C\delta,
    \]
    for a constant \(C=C(m,q)\). Indeed, the map \(\bA\mapsto \sigma_\ell(\bA)^q\) is smooth near \(\Id_m\) with value \(1\) at \(\Id_m\), and similarly for \(\bA\mapsto \widehat\bu_\ell(\bA)\) with value \(\bu_\ell\).
    Using the inequality \(|x^q-y^q|\le q(|x|^{q-1}+|y|^{q-1})|x-y|\), we obtain
    \begin{align*}
    &\big||\langle \mathbf{A}_N \mathbf{g}_i,\bu_\ell\rangle|^q-|\langle \mathbf{g}_i,\bu_\ell\rangle|^q\big|\\
    &\qquad= \big|\sigma_\ell(\mathbf{A}_N)^q|\langle \mathbf{g}_i,\widehat\bu_\ell(\mathbf{A}_N)\rangle|^q-|\langle \mathbf{g}_i,\bu_\ell\rangle|^q\big|\\
    &\qquad\le |\sigma_\ell(\mathbf{A}_N)^q-1|\,|\langle \mathbf{g}_i,\widehat\bu_\ell(\mathbf{A}_N)\rangle|^q + \big||\langle \mathbf{g}_i,\widehat\bu_\ell(\mathbf{A}_N)\rangle|^q-|\langle \mathbf{g}_i,\bu_\ell\rangle|^q\big|\\
    &\qquad\le C\delta\|\mathbf{g}_i\|_2^q
    +q\big(|\langle \mathbf{g}_i,\widehat\bu_\ell(\mathbf{A}_N)\rangle|^{q-1}+|\langle \mathbf{g}_i,\bu_\ell\rangle|^{q-1}\big)\,|\langle \mathbf{g}_i,\widehat\bu_\ell(\mathbf{A}_N)-\bu_\ell\rangle|\\
    &\qquad\le C\delta\|\mathbf{g}_i\|_2^q
    +q(1+\delta)^{q-1}(\|\mathbf{g}_i\|_2^{q-1}+\|\mathbf{g}_i\|_2^{q-1})\,\|\mathbf{g}_i\|_2\,\|\widehat\bu_\ell(\mathbf{A}_N)-\bu_\ell\|_2\\
    &\qquad\le C'\delta\,\|\mathbf{g}_i\|_2^q,
    \end{align*}
    where we used \(|\langle g,u\rangle|\le \|g\|_2\|u\|_2\) and \(\|u\|_2=1\). The constant \(C'\) depends only on \(m\) and \(q\).
    Summing over \(i\) gives
    \[
    \|\mathbf{Y}_N-\mathbf{Z}_N\|_\infty\ \le\ \frac{C'\delta}{N}\sum_{i=1}^N \|\mathbf{g}_i\|_2^q
    \]
    if $\mathbf{A}_N\in B_\delta(\Id_m)$. Therefore, for any \(\eps>0\),
    \[
    \PP(\|\mathbf{Y}_N-\mathbf{Z}_N\|_\infty>\eps)
    \le \PP\Big(\frac{1}{N}\sum_{i=1}^N \|\mathbf{g}_i\|_2^q> \frac{\eps}{C'\delta}\Big)\ +\ \PP(\mathbf{A}_N\notin B_\delta(\Id_m)).
    \]
    The second term is negligible at speed \(N^{2/q}\). For the first term, using the norm equivalence \(\|x\|_2^q = (\sum_{r=1}^m x_r^2)^{q/2} \le m^{q/2}\sum_{r=1}^m |x_r|^q\) (valid for \(q\ge 2\) by Jensen's inequality),
    \[
    \PP\Big(\frac{1}{N}\sum_{i=1}^N \|\mathbf{g}_i\|_2^q> u\Big)
    \ \le\ \sum_{r=1}^m \PP\Big(\frac{1}{N}\sum_{i=1}^N |\langle \mathbf{g}_i,\mathbf{e}_r\rangle|^q> \frac{u}{m^{q/2+1}}\Big).
    \]
    Each summand decays like \(\exp(-c(u)\,N^{2/q})\). With \(u=\eps/(C'\delta)\), the rate goes to \(\infty\) as \(\delta\downarrow0\). Thus \(\mathbf{Y}_N\) and \(\mathbf{Z}_N\) are exponentially equivalent. Since \(\mathbf{Z}_N\) satisfies the LDP lower bound, the same holds for \(\mathbf{Y}_N\).    
\end{proof}

\begin{proof}[Proof of Theorem \ref{thm: supprt function LDP p<2}]
    By \cite[Lemma 4.3]{prochno2024limit},
    \[
    \PP\big(h(N^{1/p-1/2}\bV_N\BB_p^N,\cdot)\ge f\big)
    =\PP\Big(\forall\,\bu\in\bbS^{m-1}:\ \frac{1}{N}\sum_{i=1}^N\big|\langle\sqrt{N}\,\bv_i,\bu\rangle\big|^q\ge f(\bu)^q\Big).
    \]
    Fix a dense sequence without repetitions \(\{\bu_j\}_{j\ge1}\subseteq\bbS^{m-1}\). For each \(N\) the map \(\bu\mapsto N^{-1}\sum_{i=1}^N|\langle\sqrt{N}\,\bv_i,\bu\rangle|^q\) is continuous, hence
    \[
    \Big\{\forall\,\bu:\ \frac{1}{N}\sum_{i=1}^N|\langle\sqrt{N}\,\bv_i,\bu\rangle|^q\ge f(\bu)^q\Big\}
    =\bigcap_{j\ge1}\Big\{\frac{1}{N}\sum_{i=1}^N|\langle\sqrt{N}\,\bv_i,\bu_j\rangle|^q\ge f(\bu_j)^q\Big\}.
    \]
    For \(N\in\bbN\) and \(j\ge1\) set
    \[
    A_j^{(N)}:=\Big\{\frac{1}{N}\sum_{i=1}^N\big|\langle \sqrt{N}\,\bv_i,\bu_j\rangle\big|^q\ge f(\bu_j)^q\Big\}.
    \]
    Then \(\bigcap_{j\ge1}A_j^{(N)}\subseteq \bigcap_{j=1}^k A_j^{(N)}\) for all \(k\), so
    \[
    \PP\Big(\bigcap_{j\ge1}A_j^{(N)}\Big)\le \PP\Big(\bigcap_{j=1}^k A_j^{(N)}\Big) \qquad\text{and}\qquad
    \PP\Big(\bigcap_{j\ge1}A_j^{(N)}\Big)=\lim_{k\to\infty}\PP\Big(\bigcap_{j=1}^k A_j^{(N)}\Big).
    \]
    
    \noindent \emph{Upper bound.} For fixed \(k\) apply Theorem \ref{thm: LDP for finite stiefel evals} to the \(k\)-tuple \((\bu_1,\ldots,\bu_k)\) with \(t_j=f(\bu_j)^q-M_q\). Writing \(\bSigma_{r,s}=\langle \bu_r,\bu_s\rangle\) and \(\bs=(s_1,\ldots,s_k)\), define
    \[
    \calJ_k(\bs):=\calI_{\calJ^{(k)}}\big((\bs^q-M_q\bfm1)\big),\qquad \bs\in[M_q^{1/q},\infty)^k,
    \]
    where $\calI_{\calJ^{(k)}}$ is the rate from Theorem~\ref{thm:abs-powers-gauss} with covariance matrix
    $\bSigma^{(k)}=(\langle \bu_r,\bu_s\rangle)_{1\le r,s\le k}$. Therefore,
    \[
    \limsup_{N\to\infty}\frac{1}{N^{2/q}}\log \PP\Big(\bigcap_{j=1}^k A_j^{(N)}\Big)
    \le -\,\calJ_k\big(f(\bu_1),\ldots,f(\bu_k)\big).
    \]
    Since \(\PP(\bigcap_{j\ge1}A_j^{(N)})\le \PP(\bigcap_{j=1}^k A_j^{(N)})\) for all \(k\), we obtain
    \[
    \limsup_{N\to\infty}\frac{1}{N^{2/q}}\log \PP\big(h(N^{1/p-1/2}\bV_N\BB_p^N,\cdot)\ge f\big)
    \le -\sup_{k\in\bbN} \calJ_k\big(f(\bu_1),\ldots,f(\bu_k)\big),
    \]
    which implies the desired bound.\\
    \noindent \emph{Lower bound.}
For fixed $k\in\bbN$, apply Theorem~\ref{thm: LDP for finite stiefel evals} to $(\bu_1,\ldots,\bu_k)$ with
$t_j=f(\bu_j)^q-M_q$. With $\calJ_k$ as defined above, this gives
\[
\lim_{N\to\infty}\frac{1}{N^{2/q}}\log \PP\Big(\bigcap_{j=1}^k A_j^{(N)}\Big)
= -\,\calJ_k\big(f(\bu_1),\ldots,f(\bu_k)\big).
\]
Since $\bigcap_{j\ge1}A_j^{(N)}=\bigcap_{k\ge1}\bigcap_{j=1}^kA_j^{(N)}$ and the events decrease in $k$,
\[
\PP\Big(\bigcap_{j\ge1}A_j^{(N)}\Big)=\inf_{k\in\bbN}\PP\Big(\bigcap_{j=1}^kA_j^{(N)}\Big),
\]
hence
\[
\frac{1}{N^{2/q}}\log \PP\Big(\bigcap_{j\ge1}A_j^{(N)}\Big)
=\inf_{k\in\bbN}\frac{1}{N^{2/q}}\log \PP\Big(\bigcap_{j=1}^kA_j^{(N)}\Big).
\]
Therefore,
\begin{align*}
\liminf_{N\to\infty}\frac{1}{N^{2/q}}\log \PP\Big(\bigcap_{j\ge1}A_j^{(N)}\Big)
&\ge \inf_{k\in\bbN}\ \liminf_{N\to\infty}\frac{1}{N^{2/q}}\log \PP\Big(\bigcap_{j=1}^kA_j^{(N)}\Big)\\
&= \inf_{k\in\bbN}\Big(-\calJ_k\big(f(\bu_1),\ldots,f(\bu_k)\big)\Big)\\
&= -\sup_{k\in\bbN}\calJ_k\big(f(\bu_1),\ldots,f(\bu_k)\big).
\end{align*}
    Combining the bounds yields
    \[
    \lim_{N\to\infty}\frac{1}{N^{2/q}}\log \PP\big(h(N^{1/p-1/2}\bV_N\BB_p^N,\cdot)\ge f\big)
    = - \sup_{k\in\bbN} \calJ_k\big(f(\bu_1),\ldots,f(\bu_k)\big).
    \]
\end{proof}

\begin{proof}[Proof of Corollary \ref{cor: MDP in Rd}]
    Set \(\psi(\bx,\btheta):=\bx-\btheta\). Then \(\bPsi(\btheta)=\EE[\psi(\bX,\btheta)]=\bmu-\btheta\) and \(\bPsi_N(\btheta)=N^{-1}\sum_{i=1}^N \bX_i-\btheta\). Condition (C1) holds. For (C2): \(\bPsi\) has the unique zero \(\btheta_0=\bmu\); \(\BB_\eta(\bmu)\subseteq\RR^d\) for every \(\eta>0\); the derivative at \(\bmu\) is \(-\operatorname{Id}_d\) (nonsingular); \(\bPsi\) is a homeomorphism. Moreover, by \cite[Lemma 2.5]{eichelsbacher2003moderate} and \eqref{eq: cond for MDP}, \(\EE\|\bX\|^2<\infty\). For (C3)(i),
    \[
    \frac{\sqrt{N}}{a_N}\sup_{\btheta\in\BB_\eta(\bmu)}\|\bPsi_N(\btheta)-\bPsi(\btheta)\|
    =\Big\|\frac{1}{a_N\sqrt{N}}\sum_{i=1}^N (\bX_i-\bmu)\Big\|\xrightarrow{\PP}0,
    \]
    since \(a_N\to\infty\) and by the CLT and Slutsky's theorem.
    For (C3)(ii), using that \(\sup_{\btheta\in\BB_\eta(\bmu)}\|\bX-\btheta\|\le \|\bX-\bmu\|+\eta\),
    \[
    N\,\PP\big(\sup_{\btheta\in\BB_\eta(\bmu)}\|\psi(\bX,\btheta)\|\ge \sqrt{N}\,a_N\big)
    \le N\,\PP\big(\|\bX-\bmu\|\ge \sqrt{N}\,a_N-\eta\big)
    \le N\,\PP\big(\|\bX\|\ge \frac12\sqrt{N}\,a_N\big)
    \]
    for all large \(N\). Hence, \eqref{eq: cond for MDP} implies (C3)(ii). Proposition \ref{prop: MDP general} gives an LDP with speed \(a_N^2\) and rate function
    \[
    I(\bz)=\frac12\langle \bz,\bSigma^{-1}\bz\rangle,\qquad \bSigma=\Cov(\bX).
    \]
    Finally,
    \[
    \frac{1}{a_N\sqrt{N}}\sum_{i=1}^N(\bX_i-\bmu)=\frac{\sqrt{N}}{a_N}(\btheta_N-\btheta_0),
    \]
    so the claim follows.
    \end{proof}
    
    \begin{proof}[Proof of Theorem \ref{thm:MDP}]
    Write \(x_N=\sqrt{N}\,a_N\) with \(a_N\to\infty\) and \(a_N/\sqrt{N}\to0\). It suffices, by Corollary \ref{cor: MDP in Rd} applied to \(\bX\), to verify \eqref{eq: cond for MDP} with \(a_N=x_N/\sqrt{N}\).
    By the tail assumption,
    \[
    N\,\PP\big(\|\bX\|\ge x_N\big)\le N\,e^{-c x_N^{\alpha}}
    \le \exp\big(\log N-c\,x_N^{\alpha}\big)
    \le \exp\big(-c' x_N^{\alpha}\big)
    \]
    for large \(N\) (here \(c'\in(0,c)\)), since \(\sqrt{N}/x_N\to0\) implies \(x_N^{\alpha}/\log N\to\infty\).
    Therefore,
    \[
    \limsup_{N\to\infty}\frac{N}{x_N^2}\log\Big(N\,\PP(\|\bX\|\ge x_N)\Big)
    \le -c'\,\limsup_{N\to\infty}\frac{N}{x_N^{2-\alpha}}
    = -\infty,
    \]
    because \(x_N/N^{1/(2-\alpha)}\to0\).
    Corollary \ref{cor: MDP in Rd} with \(a_N=x_N/\sqrt{N}\) yields the LDP on \(\bbR^k\) with speed \(x_N^2/N\) and good rate function \(I(\bz)=\frac12\,\bz^\top\bSigma^{-1}\bz\).
\end{proof}

\section*{Acknowledgment}
	
PT and JP are supported by the DFG project \emph{Limit theorems for the volume of random projections of $\ell_p$-balls} (project number 516672205). During the work on this paper, PT and JP visited NG at TUM. We are grateful for the hospitality and the support.



\bibliographystyle{settings/apalike2}
\bibliography{settings/biblio}

@book{dembo2009techniques,
  title={Large Deviations Techniques and Applications},
  author={Dembo, Amir and Zeitouni, Ofer},
  year={2009},
  publisher={Springer}
}

@book{kallenberg1997foundations,
  title={Foundations of modern probability},
  author={Kallenberg, Olav},
  volume={2},
  year={1997},
  publisher={Springer}
}

@article{eichelsbacher2003moderate,
  title={Moderate deviations for iid random variables},
  author={Eichelsbacher, Peter and L{\"o}we, Matthias},
  journal={ESAIM: Probability and Statistics},
  volume={7},
  pages={209--218},
  year={2003},
  publisher={EDP Sciences}
}

@article {kim2021large,
    AUTHOR = {Kim, Steven Soojin and Ramanan, Kavita},
     TITLE = {Large deviation principles induced by the {S}tiefel manifold,
              and random multidimensional projections},
   JOURNAL = {Electron. J. Probab.},
  FJOURNAL = {Electronic Journal of Probability},
    VOLUME = {28},
      YEAR = {2023},
     PAGES = {Paper No. 169, 23},
      ISSN = {1083-6489},
   MRCLASS = {60F10 (52A23 60B20)},
  MRNUMBER = {4677189},
MRREVIEWER = {Antoine\ J.\ Lejay},
       DOI = {10.1214/23-ejp1023},
       URL = {https://doi.org/10.1214/23-ejp1023},
}

@article{kim2018conditional,
  title={A conditional limit theorem for high-dimensional $\ell_p$-spheres},
  author={Kim, Steven S and Ramanan, Kavita},
  journal={Journal of Applied Probability},
  volume={55},
  number={4},
  pages={1060--1077},
  year={2018},
  publisher={Cambridge University Press}
}

@article{gao2011delta,
  title={Delta method in large deviations and moderate deviations for estimators},
  author={Gao, Fuqing and Zhao, Xingqiu},
  journal={The Annals of Statistics},
  pages={1211--1240},
  year={2011},
  publisher={JSTOR}
}

@incollection{brosset2022large,
  title={Large deviations at the transition for sums of {W}eibull-like random variables},
  author={Brosset, Fabien and Klein, Thierry and Lagnoux, Agn{\`e}s and Petit, Pierre},
  booktitle={S{\'e}minaire de Probabilit{\'e}s LI},
  pages={239--257},
  year={2022},
  publisher={Springer}
}

@article{aurzada2020large,
  title={Large deviations for infinite weighted sums of stretched exponential random variables},
  author={Aurzada, Frank},
  journal={Journal of Mathematical Analysis and Applications},
  volume={485},
  number={2},
  pages={123814},
  year={2020},
  publisher={Elsevier}
}

@article{nagaev1969integral,
	title={Integral limit theorems taking into account large deviations when {C}ram{\'e}r's condition does not hold. II},
	author={Nagaev, Aleksandr Viktorovich},
	journal={Teoriya Veroyatnostei i ee Primeneniya},
	volume={14},
	number={2},
	pages={203--216},
	year={1969},
	publisher={Russian Academy of Sciences, Steklov Mathematical Institute of Russian~…}
}

@article{chakraborty2013multivariate,
	title={On multivariate folded normal distribution},
	author={Chakraborty, Ashis Kumar and Chatterjee, Moutushi},
	journal={Sankhya B},
	volume={75},
	pages={1--15},
	year={2013},
	publisher={Springer}
}

@article{hanagal1996multivariate,
	title={A multivariate {W}eibull distribution},
	author={Hanagal, David D},
	journal={Economic Quality Control},
	volume={11},
	pages={193--200},
	year={1996},
	publisher={Citeseer}
}

@article{hagele2021large,
	title={Large deviations for a class of multivariate heavy-tailed risk processes used in insurance and finance},
	author={H{\"a}gele, Miriam and Lehtomaa, Jaakko},
	journal={Journal of Risk and Financial Management},
	volume={14},
	number={5},
	pages={202},
	year={2021},
	publisher={MDPI}
}

@article{prochno2024limit,
  title={Limit Theorems for the Volume of Random Projections and Sections of $\ell_p^N$-balls},
  author={Prochno, Joscha and Th{\"a}le, Christoph and Tuchel, Philipp},
  journal={arXiv preprint arXiv:2412.16054},
  year={2024}
}

@article{Nagaev1969a,
  author  = {Nagaev, A. V.},
  title   = {Integral limit theorems for large deviations when {Cram{\'e}r}'s condition is not fulfilled I},
  journal = {Theory of Probability and its Applications},
  year    = {1969},
  volume  = {14},
  pages   = {51--64},
}

@article{Nagaev1969b,
  author  = {Nagaev, A. V.},
  title   = {Integral limit theorems for large deviations when {Cram{\'e}r}'s condition is not fulfilled II},
  journal = {Theory of Probability and its Applications},
  year    = {1969},
  volume  = {14},
  pages   = {193--208},
}

@article{Denisov2008,
  author  = {Denisov, D. and Dieker, A. B. and Shneer, V.},
  title   = {Large deviations for random walks under subexponentiality: The big-jump domain},
  journal = {Annals of Probability},
  year    = {2008},
  volume  = {36},
  number  = {5},
  pages   = {1946--1991},
  doi     = {10.1214/07-AOP382},
}

@article{Gantert2014,
  author  = {Gantert, Nina and Ramanan, Kavita and Rembart, Franz},
  title   = {Large deviations for weighted sums of stretched exponential random variables},
  journal = {Electronic Communications in Probability},
  year    = {2014},
  volume  = {19},
  pages   = {1--14},
  doi     = {10.1214/ECP.v19-3266},
}

@article{cramer1938nouveau,
  author = {Cram{\'e}r, Harald},
  title = {Sur un nouveau th\'eor\`eme-limite de la th\'eorie des probabilit\'es},
  journal = {Actualit\'es Scientifiques et Industrielles},
  volume = {736},
  pages = {5--23},
  year = {1938}
}

@article{lehtomaa2017large,
  author = {Lehtomaa, Jaakko},
  title = {Large deviations of means of heavy-tailed random variables with finite moments of all orders},
  journal = {Journal of Applied Probability},
  volume = {54},
  number = {1},
  pages = {66--81},
  year = {2017}
}

@article{varadhan1966asymptotic,
  title={Asymptotic probabilities and differential equations},
  author={S. R. S. Varadhan},
  journal={Communications on Pure and Applied Mathematics},
  volume={19},
  number={3},
  pages={261--286},
  year={1966},
  publisher={Wiley Online Library}
}

@article{donsker1975asymptotic,
  title={Asymptotic evaluation of certain {M}arkov process expectations for large time, I},
  author={M. D. Donsker and S. R. S. Varadhan},
  journal={Communications on pure and applied mathematics},
  volume={28},
  number={1},
  pages={1--47},
  year={1975},
  publisher={Wiley Online Library}
}

@article{kabluchko2024strange,
  title={Strange shadows of $\ell_p$-balls},
  author={Kabluchko, Zakhar and Sonnleitner, Mathias},
  journal={arXiv preprint arXiv:2412.17475},
  year={2024}
}

@article{KabluchkoProchnoThaele2019,
  author        = {Kabluchko, Zakhar and Prochno, Joscha and Th{\"a}le, Christoph},
  title         = {High-dimensional limit theorems for random vectors in {$\ell_p^n$}-balls},
  journal       = {Communications in Contemporary Mathematics},
  year          = {2019},
  volume        = {21},
  number        = {1},
  pages         = {1750092},
  doi           = {10.1142/S0219199717500924},
  eprint        = {1709.09470},
  archivePrefix = {arXiv},
  primaryClass  = {math.FA}
}

@article{KabluchkoProchnoThaele2021,
  author        = {Kabluchko, Zakhar and Prochno, Joscha and Th{\"a}le, Christoph},
  title         = {High-dimensional limit theorems for random vectors in {$\ell_p^n$}-balls. {II}},
  journal       = {Communications in Contemporary Mathematics},
  year          = {2021},
  volume        = {23},
  number        = {3},
  pages         = {1950073},
  doi           = {10.1142/S0219199719500731},
  eprint        = {1906.03599},
  archivePrefix = {arXiv},
  primaryClass  = {math.PR}
}

@article{GantertKimRamanan2017,
  author        = {Gantert, Nina and Kim, Steven Soojin and Ramanan, Kavita},
  title         = {Large deviations for random projections of {$\ell^p$} balls},
  journal       = {The Annals of Probability},
  year          = {2017},
  volume        = {45},
  number        = {6B},
  pages         = {4419--4476},
  doi           = {10.1214/16-AOP1169},
  eprint        = {1512.04988},
  archivePrefix = {arXiv},
  primaryClass  = {math.PR}
}

@article{AlonsoGutierrezProchnoThaele2018,
  author        = {Alonso-Guti{\'e}rrez, David and Prochno, Joscha and Th{\"a}le, Christoph},
  title         = {Large deviations for high-dimensional random projections of {$\ell_p^n$}-balls},
  journal       = {Advances in Applied Mathematics},
  year          = {2018},
  volume        = {99},
  pages         = {1--35},
  doi           = {10.1016/j.aam.2018.04.003},
  eprint        = {1608.03863},
  archivePrefix = {arXiv},
  primaryClass  = {math.PR}
}

@article{prochno2024large,
  title={The large and moderate deviations approach in geometric functional analysis},
  author={Prochno, Joscha},
  journal={arXiv preprint arXiv:2403.03940},
  year={2024}
}

@article{MarshallOlkin1967,
  author  = {Marshall, Albert W. and Olkin, Ingram},
  title   = {A Multivariate Exponential Distribution},
  journal = {Journal of the American Statistical Association},
  volume  = {62},
  number  = {317},
  pages   = {30--44},
  year    = {1967},
  doi     = {10.1080/01621459.1967.10482885}
}

@article{JoseRisticJoseph2011,
  author  = {Jose, K. K. and Risti{\'c}, Miroslav M. and Joseph, Ancy},
  title   = {Marshall--Olkin bivariate {W}eibull distributions and processes},
  journal = {Statistical Papers},
  volume  = {52},
  pages   = {789--798},
  year    = {2011},
  doi     = {10.1007/s00362-009-0287-8}
}

@article{GomezGomezVillegasMarin1998,
  author  = {G{\'o}mez, E. and G{\'o}mez-Villegas, M. A. and Mar{\'i}n, J. M.},
  title   = {A multivariate generalization of the power exponential family of distributions},
  journal = {Communications in Statistics -- Theory and Methods},
  volume  = {27},
  number  = {3},
  pages   = {589--600},
  year    = {1998}
}

@article{DangBrowneMcNicholas2015,
  author  = {Dang, Utkarsh J. and Browne, Ryan P. and McNicholas, Paul D.},
  title   = {Mixtures of multivariate power exponential distributions},
  journal = {Biometrics},
  volume  = {71},
  number  = {4},
  pages   = {1081--1089},
  year    = {2015},
  doi     = {10.1111/biom.12351}
}

@article{VerdoolaegeScheunders2012,
  author  = {Verdoolaege, Geert and Scheunders, Paul},
  title   = {On the Geometry of Multivariate Generalized {G}aussian Models},
  journal = {Journal of Mathematical Imaging and Vision},
  volume  = {43},
  number  = {3},
  pages   = {180--193},
  year    = {2012},
  doi     = {10.1007/s10851-011-0297-8}
}

@inproceedings{WainwrightSimoncelli1999,
  author    = {Wainwright, Martin J. and Simoncelli, Eero P.},
  title     = {Scale Mixtures of {G}aussians and the Statistics of Natural Images},
  booktitle = {Advances in Neural Information Processing Systems},
  volume    = {12},
  pages     = {856--862},
  year      = {1999}
}

@article{HashorvaPakesTang2010,
  author  = {Hashorva, Enkelejd and Pakes, Anthony G. and Tang, Qihe},
  title   = {Asymptotics of random contractions},
  journal = {Insurance: Mathematics and Economics},
  volume  = {47},
  number  = {3},
  pages   = {405--414},
  year    = {2010},
  doi     = {10.1016/j.insmatheco.2010.08.006}
}

@article{Latala2006,
  author  = {Lata{\l}a, Rafa{\l}},
  title   = {Estimates of moments and tails of {G}aussian chaoses},
  journal = {The Annals of Probability},
  volume  = {34},
  number  = {6},
  pages   = {2315--2331},
  year    = {2006},
  doi     = {10.1214/009117906000000421}
}

@article{KuchibhotlaChakrabortty2018,
  author       = {Kuchibhotla, Arun Kumar and Chakrabortty, Abhishek},
  title        = {Moving Beyond Sub-{G}aussianity in High-Dimensional Statistics: Applications in Covariance Estimation and Linear Regression},
  year         = {2018},
  eprint       = {1804.02605},
  archivePrefix= {arXiv}
}

@article{KabluchkoProchnoThale2019,
  author    = {Kabluchko, Zakhar and Prochno, Joscha and Th{\"a}le, Christoph},
  title     = {A New Look at Random Projections of the Cube and General Product Measures},
  journal   = {Bernoulli},
  volume    = {27},
  number    = {3},
  pages     = {2117--2138},
  year      = {2021},
  doi       = {10.3150/20-BEJ1303}
}

\excludePart{
\begin{appendix}
\section{Appendix}
\end{appendix}
}
\end{document}